\documentclass{amsart}

\usepackage[all]{xy}

\usepackage{pictexwd, epsfig, amsthm, amssymb, amsmath, graphicx, psfrag,amsxtra, latexsym}
\usepackage{enumerate}

\DeclareMathOperator{\Sym}{Sym} 
\DeclareMathOperator{\rank}{rank}

\DeclareMathOperator{\hei}{Hei}
 
 \DeclareMathOperator{\G}{\Gamma}

 \DeclareMathOperator{\coker}{coker}
 
 \DeclareMathOperator{\aut}{Aut}

\theoremstyle{plain}
\newtheorem{Thm}{Theorem}
\newtheorem{lemma}[Thm]{Lemma}
\newtheorem{corollary}[Thm]{Corollary}

\newtheorem{example}[Thm]{Example}

\theoremstyle{definition}
\newtheorem{definition}[Thm]{Definition}

\theoremstyle{remark}

\begin{document}

\title[Rational equivariant $K$-homology of certain $3$-orbifold groups.]
      {Rational equivariant $K$-homology of low dimensional groups}

\author{J.-F.\ Lafont}
\address{Department of Mathematics\\
         Ohio State University\\
         Columbus, OH  43210\\
         USA}
\email[Jean-Fran\c{c}ois\ Lafont]{jlafont@math.ohio-state.edu}
\author{I.\ J.\ Ortiz}
\address{Department of Mathematics\\
         Miami University\\
         Oxford, OH 45056\\
         USA}
\email[Ivonne J.\ Ortiz]{ortizi@muohio.edu}
\author{R. J. S\'anchez-Garc\'ia}
\address{School of Mathematics\\
University of Southampton\\
Southampton SO17 1BJ\\ 
UK}
\email[Rub\'en J. S\'anchez-Garc\'ia]{R.Sanchez-Garcia@soton.ac.uk}

\begin{abstract}
We consider groups $G$ which have a cocompact, $3$-manifold model for the classifying space $\underline{E}G$.
We provide an algorithm for computing the rationalized equivariant $K$-homology of $\underline{E}G$. Under the additional
hypothesis that the quotient $3$-orbifold $\underline{E}G/G$ is geometrizable, 
the rationalized $K$-homology groups coincide with the groups $K_*(C^*_{red}G)\otimes \mathbb Q$. 
We illustrate our algorithm on some concrete examples.

\end{abstract}

\maketitle

\section{Introduction}

We consider groups $G$ which have a cocompact, $3$-manifold model for the classifying space $\underline{E}G$. 
For such groups, we are interested in computing the equivariant $K$-homology of $\underline{E}G$. 
We develop an algorithm to compute the {\it rational}
equivariant $K$-homology groups. If in addition we assume that the quotient $3$-orbifold $\underline{E}G/G$ 
is geometrizable,  then $G$ satisfies
the Baum-Connes conjecture, and the rational equivariant $K$-homology groups coincide with 
the groups $K_*(C^*_{red}G)$. These are the rationalized (topological) $K$-theory groups of the 
reduced $C^*$-algebra of $G$. 

Some general recipes exist for computing the rational $K$-theory
of an arbitrary group (see L\"uck and Oliver \cite{LuO}, as well as L\"uck \cite{Lu1}, \cite{Lu2}). These
general recipes pass via the Chern character. They typically involve identifying certain 
conjugacy classes of cyclic subgroups, their centralizers, and certain (group) homology computations.

In contrast, our methods rely instead on the low-dimensionality of the model for the classifying space
$\underline{E}G$. Given a description of the model space $\underline{E}G$,
our procedure is entirely algorithmic, and returns the ranks of the $K$-homology groups. 

Let us briefly outline the contents of this paper. In Section 2, we 
provide some background material. Section 3 is devoted to explaining our algorithm, and
the requisite proofs showing that the algorithm gives the desired $K$-groups. In Section 4 
we implement our algorithm on several concrete classes of examples. 
Section 5 has some concluding remarks.



\section{Background material}

\subsection{$C^*$-algebra}

Given any discrete group $G$, one can form the associated reduced $C^*$-algebra. 
This Banach algebra is obtained by looking at the action $g\mapsto \lambda_g$ of $G$ on the Hilbert space
 $l^2(G)$ of square summable complex-valued functions on $G$, given by the left regular representation:
$$\lambda_g \cdot f(h) = f\big(g^{-1}h\big) \hskip 20pt g, h \in G, \hskip 20pt f\in l^2(G).$$
The algebra $C^*_r(G)$ is defined to be the operator norm closure of the linear span of the operators $\lambda_g$ inside the
space $B\big(l^2(G)\big)$ of bounded linear operators on $l^2(G)$. The Banach algebra $C^*_r(G)$ encodes various 
analytic properties of the group $G$. 

\subsection{Topological $K$-theory}

For a $C^*$-algebra $A$, the corresponding (topological) $K$-theory groups can be defined in the following manner.
The group $K_0(A)$ is defined to be the Grothendieck completion of the semi-group of finitely generated projective 
$A$-modules (with group operation given by direct sum). Since the algebra $A$ comes equipped with a topology, one
has an induced topology on the space $GL_n(A)$ of invertible $(n\times n)$-matrices with entries in $A$, and as such
one can consider the group $\pi_0\big(GL_n(A)\big)$ of connected components of $GL_n(A)$ (note that this is indeed 
a group, not just a set). The group $K_1(A)$ is defined to be $\lim \pi_0\big(GL_n(A)\big)$, where the limit is taken with
respect to the sequence of natural inclusions of $GL_n(A)\hookrightarrow GL_{n+1}(A)$. The higher $K$-theory groups
$K_q(A)$ are similarly defined to be $\lim \pi_{q-1}\big(GL_n(A)\big)$, for $q\geq 2$. Alternatively, one can identify
the functors $K_q(A)$ for all $q\in \mathbb Z$ via Bott $2$-periodicity in $q$, i.e.~$K_q(A)\cong K_{q+2}(A)$ for all 
$q$.

\subsection{Baum-Connes conjecture}

Let us now recall the statement of the Baum-Connes conjecture (see \cite{BCH}, \cite{DL}). 
Given a discrete group $G$, there exists 
a specific generalized equivariant homology theory having the property that, if one evaluates it on a point
$*$ with trivial $G$-action, the resulting homology groups satisfy $H_n^G(*)\cong K_n(C^*_r(G))$. Now for
any $G$-CW-complex $X$, one has an obvious equivariant map $X \rightarrow *$. It follows
from the basic properties of equivariant homology theories that there is an induced {\it assembly map}:
$$H_n^G(X) \rightarrow H_n^G(*)\cong K_n(C^*_r(G)).$$
Associated to a discrete group $G$, we have a classifying space for proper actions $\underline{E}G$. The
$G$-CW-complex $\underline{E}G$ is well-defined up to $G$-equivariant homotopy equivalence, and is 
characterized by the following two properties:
\begin{itemize}
\item if $H\leq G$ is any infinite subgroup of $G$, then $\underline{E}G^H =\emptyset$, and
\item if $H\leq G$ is any finite subgroup of $G$, then $\underline{E}G^H$ is contractible.
\end{itemize}
The Baum-Connes conjecture states that the assembly map 
$$H_n^G(\underline{E}G) \rightarrow H_n^G(*)\cong K_n(C^*_r(G))$$
corresponding to $\underline{E}G$ is an isomorphism. 
For a thorough discussion of this topic, we refer the reader to the book by Mislin and Valette \cite{MV} or the survey article 
by L\"uck and Reich \cite{LuR}.

\subsection{$3$-orbifold groups}\label{fund-dom}
We are studying groups $G$ having a cocompact $3$-manifold model for $\underline{E}G$. 
Let $X$ denote this specific model for the classifying
space, and for this section, we will further assume the 
quotient $3$-orbifold $X/G$ is {\it geometrizable}.

The validity of the Baum-Connes conjecture for fundamental groups of orientable $3$-manifolds
has been established by Matthey, Oyono-Oyono, and Pitsch \cite[Thm. 1.1]{MOP}
(see also \cite[Thm. 5.18]{MV} or \cite[Thm. 5.2]{LuR}). The same argument works in the context
of geometrizable $3$-orbifolds. We provide some details for the convenience of the reader.

\begin{lemma}\label{B-C-conj}
The Baum-Connes conjecture holds for the orbifold fundamental group of geometrizable $3$-orbifolds. 
\end{lemma}

\begin{proof}
In fact, the stronger {\it Baum Connes property with coefficients} holds for this class of groups. This property 
states that a certain assembly map, associated to a $G$-action on a separable $C^*$-algebra $A$, is an 
isomorphism (and recovers the classical Baum-Connes conjecture when $A= \mathbb C$). The coefficients
version has better inheritance properties, and in particular, is known to be inherited under graph of groups
constructions (amalgamations and HNN-extensions), see Oyono-Oyono \cite[Thm. 1.1]{O-O}. 

The orbifold fundamental group of a geometrizable $3$-orbifold can be expressed as an iterated graph of 
groups, with all initial vertex groups being orbifold fundamental groups of geometric $3$-orbifolds. Geometric
$3$-orbifolds are cofinite volume quotients of one of
the eight $3$-dimensional geometries. Combined with Oyono-Oyono's result, the Lemma reduces
to establishing the property for the orbifold fundamental group of finite volume geometric $3$-orbifolds.

The fundamental work of Higson and
Kasparov \cite{HK} established the Baum-Connes property with coefficients for all groups satisfying 
the Haagerup property. We refer the reader to the monograph \cite{CCJJV} for a detailed exposition to 
the Haagerup property. We will merely require the fact that groups acting with cofinite volume on all
eight $3$-dimensional geometries ($\mathbb E ^3$, $S^3$, $S^2\times \mathbb E^1$, $\mathbb H^3$, 
$\mathbb H^2 \times \mathbb E^1$, 
$\widetilde {PSL}_2(\mathbb R)$, $Nil$, and $Sol$) always have the Haagerup property, which will 
conclude the proof of the Lemma.

For the five geometries $\mathbb E ^3$, $S^3$, $S^2\times \mathbb E^1$, $Nil$, and $Sol$, any group 
acting on these will be amenable, and hence satisfy the Haagerup property. Lattices inside groups locally
isomorphic to $SO(n,1)$ are Haagerup (see \cite[Thm. 4.0.1]{CCJJV}), and hence groups acting on the
two geometries $\mathbb H^3$ and $\widetilde {PSL}_2(\mathbb R)$ are Haagerup. Finally, the Haagerup
property is inherited by amenable extensions of Haagerup groups (see \cite[Example 6.1.6]{CCJJV}). This
implies that groups acting on $\mathbb H^2\times \mathbb E^1$ are Haagerup, for any such group is a finite
extension of a group which splits as a product of $\mathbb Z$ with a lattice in $SO(2,1)$. This concludes
the proof of the Lemma.
\end{proof}

\noindent {\bf Remark:} If one assumes that the $G$-action is smooth and orientation preserving, then
Thurston's geometrization conjecture (now a theorem) predicts that $X/G$ is a geometrizable $3$-orbifold.
The proof of the orbifold version of the conjecture was originally outlined 
by Thurston, and was independently 
established by Boileau, Leeb, and Porti \cite{BLP} and Cooper, Hodgson, an Kerckhoff \cite{CHK} (both
loosely following Thurston's approach). The manifold version of the conjecture (i.e.~trivial isotropy groups)
is of course due to the recent work of Perelman. 

\vskip 5pt

\noindent {\bf Remark:} If the quotient space $X/G$ is not known to be geometrizable (for instance, if the
$G$-action is not smooth, or does not preserve the orientation), then the argument in Lemma \ref{B-C-conj} 
does not apply. Nevertheless, our algorithm
can still be used to compute the rational equivariant $K$-homology of $\underline{E}G$. It is however 
no longer clear that this coincides with $K(C^*_r(G))\otimes \mathbb Q$.

\subsection{Polyhedral CW-structures}\label{polyhedral}
Let us briefly comment on the $G$-CW-structure of $X$. As the quotient space $X/G$ is a connected $3$-orbifold,
we can assume without loss of generality that 
the CW-structure contains a single orbit of $3$-cell. Taking a representative $3$-cell $\sigma$ 
for the unique $3$-cell orbit, we observe that the closure of $\sigma$ must {\it contain} representatives
of each lower dimensional orbit of cells. Indeed, if some lower dimensional cell had no orbit representatives
contained in $\bar \sigma$, then there would be points in that lower dimensional cell with no neighborhood
homeomorphic to $\mathbb R^3$. Pulling back the $2$-skeleton of the CW-structure via the attaching
map of the $3$-cell $\sigma$, we obtain (i) a decomposition of the $2$-sphere into the pre-images of the
individual cells, and (ii) an equivalence
relation on the $2$-sphere, identifying together points which have the same image under the attaching map. 
We note that the quotient space $X/G$ can be reconstructed from this data. If in addition we know the isotropy
subgroups of points, then $X$ itself can be reconstructed from $X/G$. We will 
assume that we are given the $G$-action on $X$, in the form of a partition and equivalence relation on the
$2$-sphere as above, along with the isotropy data.

In some cases, one can find a $G$-CW-structure which is particularly simple: the
$2$-sphere coincides with the boundary of a polyhedron, the partition of the $2$-sphere is into the faces
of the polyhedron, and the equivalence relation linearly
identifies together faces of the polyhedron. More precisely, we make the:

\begin{definition}
A {\it polyhedral} CW-structure is a CW-structure where each cell is identified with the interior of a polyhedron 
$P_i\cong \mathbb D^k$,
and the attaching maps from the boundary $\partial \mathbb D^k \cong \partial P_i$ of 
a $k$-cell to the $(k-1)$-skeleton, when restricted to each
$s$-dimensional face of $\partial P_i$, is a combinatorial homeomorphism onto an $s$-cell in the
$(k-1)$-skeleton.
\end{definition}

In the case where there is a polyhedral $G$-CW-structure on $X$ with a single $3$-cell orbit, then our
algorithms are particularly easy to implement. All the concrete examples we will see in Section 4 come 
equipped with a polyhedral $G$-CW-structure. 

\vskip 10pt

\noindent {\bf Remark:} It seems plausible that, if a $G$-CW-structure exists for a (topological) $G$-action on a 
$3$-manifold $X$, then a polyhedral
$G$-CW-complex structure should also exist. It also seems likely that, if a polyhedral $G$-CW-structure exists,
then the $G$-action on the $3$-manifold $X$ should be smoothable. 

For some concrete examples of  polyhedral $G$-CW-structures, consider the case where $X$ is either 
hyperbolic space $\mathbb H^3$ or Euclidean space $\mathbb R^3$, and the $G$-action is via isometries. Then the desired 
$G$-equivariant polyhedral CW-complex structure can be obtained by picking a suitable point $p\in X$, and considering
the Voronoi diagram with respect to the collection of points in the orbit $G\cdot p$. Another example, where $X$ is the
$3$-dimensional Nil-geometry is discussed in Section \ref{Nil-mfld-example}. 

\vskip 10pt



\section{The algorithm} \label{algorithm}

In this section, we describe the algorithm used to perform our computations. Throughout this section, 
let $G$ be a group with a smooth action on a $3$-manifold, providing a model for $\underline{E}G$.
We will assume that $\underline{E}G$ supports a polyhedral $G$-CW-structure, and that 
$P$ is a fundamental domain for the $G$-action on $X$, as described in Section \ref{polyhedral}.
So $P$ is the polyhedron corresponding to the single $3$-cell orbit, and the orbit space $\underline{B}G$
is obtained from $P$ by identifying various boundary faces together. 
We emphasize that the polyhedral $G$-CW-structure assumption serves only to facilitate the exposition:
the algorithm works equally well with an arbitrary $G$-CW-structure.


\subsection{Spectral sequence analysis}
As explained in the previous section, the Baum-Connes conjecture provides an isomorphism:
$$H_n^G(\underline{E}G) \rightarrow H_n^G(*)\cong K_n(C^*_r(G)).$$
We are interested in computing the equivariant homology 
group arising on the left hand side of the assembly map. Since our group $G$ is $3$-dimensional, 
we will let $X$ denote the $3$-dimensional manifold model for $\underline{E}G$. 
To compute the equivariant 
homology of $X$, one can use an Atiyah-Hirzebruch spectral sequence. Specifically,
there exists a spectral sequence (see \cite{DL}, or \cite[Section 8]{Q}), converging to the group $H_n^G(X)$,
with $E^2$-terms obtained by taking the homology of the following chain complex:
\begin{equation}\label{chain-clx}
\ldots \rightarrow \bigoplus _{\sigma\in (X/G)^{(p+1)}} K_q\big(C^*_r(G_{\sigma})\big)
\rightarrow \bigoplus _{\sigma\in (X/G)^{(p)}} K_q\big(C^*_r(G_{\sigma})\big)
\rightarrow \ldots
\end{equation}
In the above chain complex, $(X/G)^{(i)}$ consists of $i$-dimensional cells in the quotient $X/G$,
or equivalently, $G$-orbits of $i$-dimensional cells in $X$. The groups 
$G_\sigma$ denote the stabilizer of a cell in the orbit $\sigma$. Since our space $X$ is $3$-dimensional,
we see that our chain complex can only have non-zero terms in the range $0 \leq p \leq 3$
(the morphisms in the chain complex will be described later, see Section \ref{h3:section}). Moreover, 
since $X$ is a model for 
$\underline{E}G$, all the cell stabilizers $G_\sigma$ must be finite subgroups of $G$. For $F$ a finite group,
the groups $K_q\big(C^*_r(F)\big)$ are easy to compute:
$$
K_q\big(C^*_r(F)\big) = 
\begin{cases} 0 & \text{if $q$ is odd,} \\
\mathbb Z ^ {c(F)} & \text{if $q$ is even.}
\end{cases}
$$
Here, $c(F)$ denotes the number of conjugacy classes of elements in $F$. In fact, for $q$ even, 
$K_q\big(C^*_r(F)\big)$ can be identified with the complex representation ring of $F$.
This immediately tells us that $E^2_{pq} = 0$ for $q$ odd. 
We will denote by $\mathcal C$ the chain complex in equation (\ref{chain-clx}) corresponding to the case 
where $q$ is even. 
By the discussion above, we know that $H_p(\mathcal C)=0$ except possibly in the range 
$0 \leq p \leq 3$. We summarize this discussion in the:

\vskip 10pt

\noindent {\bf Fact 1:} The only potentially non-vanishing terms on the $E^2$-page (and hence any $E^k$-page, 
$k\geq 2$)  occur when $0\leq p \leq 3$ {\em and} $q$ is even.

\vskip 10pt

Next we note that the differentials on the $E^k$-page of the spectral sequence have bidegree $(-k, k+1)$, i.e. are
of the form 
$d^k_{p,q}: E^k_{p,q}\rightarrow E^k_{p-k, q+k-1}$. When $k=2$, alternating rows on the $E^2$-page are zero
(see {\bf Fact 1}), which implies that $E^3_{p,q} = E^2_{p,q}$. When $k=3$,
the differentials $d^3_{p,q}$ shift horizontally by three units, and up by two units. So the only potentially
non-zero differentials on the $E^3$-page are (up to vertical translation by the $2$-periodicity in $q$) those of the 
form:
$$d^3_{3,0}:  E^2_{3,0} \cong E^3_{3,0} \rightarrow E^3_{0,2} \cong E^2_{0,0}.$$

Once we have $k\geq 4$, the
differentials $d^k_{p,q}$ shift horizontally by $k\geq 4$ units.
But {\bf Fact 1} tells us that the only non-zero terms occur in the vertical strip $0\leq p \leq 3$, which forces 
$E^4_{p,q} \cong E^5_{p,q} \cong \ldots$ for all $p,q$. 
In other words, the spectral sequence collapses at the $E^4$-stage. Since the $E^2$-terms are
given by the homology of $\mathcal C$, this establishes:

\begin{lemma}\label{spec-seq-collapse}
The groups $K_q\big(C^*_r(G)\big)$ can be computed from the $E^4$-page of the spectral sequence, and coincide
with
$$
K_q\big(C^*_r(G)\big) = 
\begin{cases} H_{1}(\mathcal C) \oplus \ker (d^3_{3,0})& \text{if $q$ is odd,} \\
\coker (d^3_{3,0}) \oplus H_{2}(\mathcal C) & \text{if $q$ is even,}
\end{cases}
$$
where $d^3_{3,0}: H_{3}(\mathcal C) \rightarrow H_{0}(\mathcal C)$ is the differential appearing
on the $E^3$-page of the spectral sequence.
\end{lemma}

Since we are only interested in the rationalized equivariant $K$-homology, we can actually ignore the presence of any differentials: 
after tensoring with $\mathbb Q$ the Atiyah-Hirzebruch spectral sequence collapses at the $E^2$-page \cite[Remark 3.9]{Lu1}. 
Thus for $q$ even,
\[
	K_q\big(C^*_r(G)\big)\otimes \mathbb Q \cong \left(E^2_{0,q}\otimes \mathbb Q\right) \oplus \left( E^2_{2,q-2}\otimes 
	\mathbb Q \right) \cong \left(H_{0}(\mathcal C) \otimes \mathbb Q\right) \oplus \left( H_{2}(\mathcal C) \otimes \mathbb Q\right),
\]
and for $q$ odd,
\[
	K_q\big(C^*_r(G)\big)\otimes \mathbb Q \cong \left(E^2_{1,q-1}\otimes \mathbb Q\right) \oplus \left( E^2_{3,q-3}\otimes 
	\mathbb Q \right) \cong \left(H_{1}(\mathcal C) \otimes \mathbb Q\right) \oplus \left( H_{3}(\mathcal C) \otimes \mathbb Q\right).
\]
\begin{lemma}\label{rational-expression}
The rank of the groups $K_q\big(C^*_r(G)\big)\otimes \mathbb Q$ are given by
$$
\rank \Big(K_q\big(C^*_r(G)\big)\otimes \mathbb Q \Big) = 
\begin{cases} \rank \big(H_{1}(\mathcal C)\otimes \mathbb Q \big) + \rank \big(H_{3}(\mathcal C)\otimes \mathbb Q \big) & 
\text{if $q$ is odd,} \\
\rank \big(H_{0}(\mathcal C)\otimes \mathbb Q \big) + \rank \big(H_{2}(\mathcal C)\otimes \mathbb Q \big) & \text{if $q$ is even.}
\end{cases}
$$
\end{lemma}

\vskip 5pt

\noindent{\bf Remark:} Alternatively this result follows directly from the equivariant Chern character being a rational isomorphism \cite[Thm.~6.1]{MV}.

\vskip 5pt

In the next four sections, we explain how to algorithmically compute the ranks of the four groups 
appearing in Lemma \ref{rational-expression}.


\subsection{1-skeleton of $X/G$ and the group $H_0(\mathcal C)$}

For the group $H_{0}(\mathcal C)$, we make use of the result from \cite[Theorem 3.19]{MV}. For the convenience
of the reader, we restate the theorem:

\begin{Thm}
For $G$ an arbitrary group, we have
$$H_{0}(\mathcal C) \otimes \mathbb Q \cong \mathbb Q^{cf(G)},$$ 
where $cf(G)$ denotes the number of conjugacy classes of elements of finite order in the group $G$.
\end{Thm}

This reduces the computation of the rank of $H_{0}(\mathcal C) \otimes \mathbb Q$ to finding some
algorithm for computing the number $cf(G)$. We now explain how one can 
compute the integer $cf(G)$ in terms of the $1$-skeleton of the space $X/G$. 

\vskip 5pt

For each cell
$\sigma$ in $\underline{B}G$, we fix a reference cell $\tilde \sigma \in \underline{E}G$, having the property that
$\tilde \sigma$ maps to $\sigma$ under the quotient map $p: \underline{E}G \rightarrow \underline{B}G$.
Associated to each cell $\sigma$
in $\underline{B}G$, we have a finite subgroup $G_{\tilde \sigma} \leq G$, which is just the stabilizer of the fixed
pre-image $\tilde \sigma \in \underline{E}G$. Since the stabilizers of two distinct lifts $\tilde \sigma$, 
$\tilde \sigma ^\prime$ of the cell $\sigma$ are conjugate subgroups inside $G$, we note that the {\it conjugacy class}
of the finite subgroup $G_{\tilde \sigma}$ is independent of the choice of lift $\tilde \sigma$, and depends solely
on the cell $\sigma\in\underline{B}G$.
Now given a cell $\sigma$ in $\underline{B}G$ with a boundary 
cell $\tau$, we have associated lifts $\tilde \sigma$, $\tilde \tau$. Of course, the lift $\tilde \tau$ might not lie in
the boundary of $\tilde \sigma$, but there exists some other lift $\tilde \tau ^\prime$ of $\tau$ which {\it does} lie in
the boundary of $\tilde \sigma$. Clearly, we have an inclusion $G_{\tilde \sigma} \hookrightarrow G_{\tilde \tau ^\prime}$.
Fix an element $g_{\sigma, \tau} \in G$ with the property that $g_{\sigma, \tau}$ maps the lift $\tilde \tau ^\prime$
to the lift $\tilde \tau$. This gives us a map $\phi_{\sigma}^\tau:G_{\tilde \sigma} \hookrightarrow G_{\tilde \tau}$, obtained by
composing the inclusion $G_{\tilde \sigma} \hookrightarrow G_{\tilde \tau ^\prime}$ with the isomorphism 
$G_{\tilde \tau ^\prime} \rightarrow G_{\tilde \tau}$ given by conjugation by $g_{\sigma, \tau}$. Now the map 
$\phi_{\sigma}^\tau$ isn't well-defined, as there are different possible choices for the element $g_{\sigma, \tau}$.
However, if $g_{\sigma, \tau}^\prime$ represents a different choice of element, then since both elements
$g_{\sigma, \tau}, g_{\sigma, \tau}^\prime$ map $\tilde \tau ^\prime$ to $\tilde \tau$, we see that the product
$\big(g_{\sigma, \tau}^\prime\big) \big(g_{\sigma, \tau})^{-1}$ maps $\tilde \tau$ to itself, and hence we
obtain the equality $g_{\sigma, \tau}^\prime = h \cdot g_{\sigma, \tau}$, where $h\in G_{\tilde \tau}$.
This implies that the map $\phi_{\sigma}^\tau$ is well-defined, up to post-composition by an inner automorphism
of $G_{\tilde \tau}$.

Consider the set $F(G)$ consisting of the disjoint union of the finite groups $G_{\tilde v}$ where
$v$ ranges over vertices in the $0$-skeleton $(\underline{B}G)^{(0)}$ of $\underline{B}G$. 
Form the smallest equivalence relation $\sim$ on $F(G)$ with the property that:
\begin{enumerate}[(i)]
\item for each vertex $v\in (\underline{B}G)^{(0)}$, and elements $g, h\in G_{\tilde v}$ which are conjugate 
within $G_{\tilde v}$, we have $g\sim h$, and
\item for each edge $e\in (\underline{B}G)^{(1)}$ joining vertices $v, w\in (\underline{B}G)^{(0)}$, and element
$g\in G_{\tilde e}$, we have $\phi_e^v(g)\sim \phi_e^w(g)$.
\end{enumerate}
Note that, although the maps $\phi_{\sigma}^\tau$ are not well-defined, the equivalence relation given above {\it is}
well-defined. Indeed, for any given edge $e\in (\underline{B}G)^{(1)}$, the maps $\phi_e^v$, $\phi_e^w$ are only well-defined
up to inner automorphisms of $G_{\tilde v}, G_{\tilde w}$. In view of property (i), the resulting property (ii) is independent 
of the choice of representatives $\phi_e^v$, $\phi_e^w$.

For a finitely generated group, we let $eq(G)$ denote the number of $\sim$ equivalence classes on the corresponding
set $F(G)$. We can now establish:

\begin{lemma}\label{rank-K0}
For $G$ an arbitrary finitely generated group, we have $cf(G) = eq(G)$.
\end{lemma}

\begin{proof}
Let us write $g \approx h$ if the elements $g$, $h$ are conjugate in $G$. As each element in $F(G)$ is also an element in $G$, 
we now have the two equivalence relations $\sim, \approx$ on the set $F(G)$. It is immediate from the definition 
that $g\sim h$ implies $g\approx h$. 

Next, we argue that, for elements $g,h \in F(G)$, $g\approx h$ implies $g\sim h$. 
To see this, assume that $k \in G$ is a conjugating element, so $g= khk^{-1}$. For the action on
$\underline{E}G$, we know that $g,h$ fix vertices $\tilde v, \tilde w$ (respectively) in the $0$-skeleton $(\underline{E}G)^{(0)}$,
which project down to vertices $v,w \in (\underline{B}G)^{(0)}$ (respectively). 
Since
$g= khk^{-1}$, we also have that $g$ fixes the vertex $k\cdot \tilde w$. The $g$ fixed set $\underline{E}G^g$ is contractible, 
so we can find a path joining $\tilde v$ to $k\cdot \tilde w$ inside the subcomplex $\underline{E}G^g$. Within this 
subcomplex, we can push
any path into the $1$-skeleton, giving us a sequence of consecutive edges within the graph 
$\big(\underline{E}G^g\big) ^{(1)} \subseteq
(\underline{E}G)^{(1)}$ joining $\tilde v$ to $k\cdot \tilde w$. This projects down to a path in $(\underline{B}G)^{(1)}$ 
joining the vertex $v$ to the vertex $w$ (as $k\cdot \tilde w$ and $\tilde w$ lie in the same $G$-orbit, 
they have the same projection). Using property (ii),
the projected path gives a sequence of elements $g=g_0\sim g_1 \sim \ldots \sim g_k= h$, where each pair $g_i, g_{i+1}$ 
are in the groups associated to consecutive vertices in the path. 

So we now have that the two equivalence relations $\sim$ and $\approx$ coincide on the set $F(G)$, and in particular, have the
same number of equivalence classes. Of course, the number of $\sim$ equivalence classes is precisely the number $eq(G)$. On
the other hand, any element of finite order $g$ in $G$ must have non-trivial fixed set in $\underline{E}G$. Since the action
is cellular, this forces the existence of a fixed vertex $\bar v\in (\underline{E}G)^{(0)}$ (which might not be unique). The vertex 
$\bar v$ has an image vertex $v\in (\underline{B}G)^{(0)}$ under the quotient map, and hence $g \approx \tilde g$ for
some element $\tilde g$ in the set $F(G)$, corresponding to the subgroup $G_{\tilde v}$. This implies that the number of 
$\approx$ equivalence classes in $F(G)$ is equal to $cf(G)$, concluding the proof.
\end{proof}

\vskip 5pt

\noindent {\bf Remark:} The procedure we described in this section works for any model for $\underline{B}G$, and
would compute the $\beta_0$ of the corresponding chain complex. 
On the other hand, if one has a model for $\underline{E}G$ with the property that the quotient $\underline{B}G$ has 
few vertices and edges,
then it is fairly straightforward to calculate the number $eq(G)$ from the $1$-skeleton of $\underline{B}G$. For the groups we are
considering, we can use the model space $X$. The $1$-skeleton of $\underline{B}G$ is then a quotient of the $1$-skeleton of the 
polyhedron $P$. Along with Lemma \ref{rank-K0},
this allows us to easily compute the rank of $H_0(\mathcal C)\otimes \mathbb Q$ for the groups within our class.


\subsection{Topology of $X/G$ and the group $H_3(\mathcal C)$}\label{h3:section}

Our next step is to understand the rank of the group $H_3(\mathcal C)\otimes \mathbb Q$; this requires 
an understanding of the differentials appearing in the chain complex $\mathcal C$. In $X$, if
we have a $k$-cell $\sigma$ contained in the closure of a $(k+1)$-cell $\tau$, then we have a natural inclusion
of stabilizers $G_\tau \hookrightarrow G_\sigma$ (well defined up to conjugation in $G_\sigma$). 
Applying the functor $K_q\big(C^*_r( - )\big)$,
where $q$ is even, we get an induced morphism from the complex representation ring of $G_\tau$
to the complex representation ring of $G_\sigma$. Concretely, the image of a complex representation
$\rho$ of $G_\tau$ under this morphism is the induced complex representation $\rho \uparrow := 
Ind_{G_\tau}^{G_\sigma}\rho$ in $G_\sigma$, with multiplicity given (as usual) by the degree of the
attaching map from the boundary sphere $S^{k-1} = \partial \tau$ to the sphere $S^{k-1}=\sigma / \partial \sigma$.
Note that conjugate representations induce up to the
same representation.

In the chain complex, the individual terms are indexed
by {\it orbits} of cells in $X$, rather than individual cells. To see what the chain map does, pick an orbit of 
$(k+1)$-cells, and fix an oriented representative $\tau$. Then for each orbit of a $k$-cell, one can look at the 
$k$-cells in that oriented orbit that are incident to $\tau$, call them $\sigma_1, \ldots , \sigma_r$. The stabilizer
of each of the $\sigma_i$ is a copy of the same group $G_\sigma$ (where the identification between these 
groups is well-defined up to inner automorphisms). For each of these $\sigma_i$, the
discussion in the previous paragraph allows us to obtain a map on complex representation rings. 
Finally, one identifies the groups $G_{\sigma_i}$ with the 
group $G_{\sigma}$, and take the sum of the maps on the complex representation rings. This
completes the description of the chain maps in the complex $\mathcal C$.

\vskip 5pt

Consider a representative $\sigma$ for the single $3$-cell orbit  in the $G$-CW-complex $X$
(we can identify $\sigma$ with the interior of the polyhedron $P$). 
The stabilizer of $\sigma$ must be trivial (as any element stabilizing $\sigma$ must stabilize all of $X$). 
We conclude that 
$\mathcal C_3 = \bigoplus _{\sigma\in (X/G)^{(3)}} K_q\big(C^*_r(G_{\sigma})\big) \cong \mathbb Z$, and the
generator for this group is given by the trivial representation of the trivial group. But inducing up the trivial
representation of the trivial group always gives the left regular representation, which is just the sum of all
irreducible representations. This tells us that, for each $2$-cell in the boundary of $\sigma$, the corresponding
map on the $K$-group is non-trivial. 

Now when looking at the chain complex, the target of the differential is indexed by {\it orbits} of $2$-cells, rather
than individual $2$-cells. Each $2$-cell orbit has either one or two representatives lying in the boundary of $\sigma$.
Whether there is one or two can be decided as follows: look at the $G$-translate $\sigma ^\prime$ of $\sigma$ which is
adjacent to $\sigma$ across the given boundary $2$-cell $\tau$. Since $X$ is a {\it manifold} model for $\underline{E}G$,
there is a unique such $\sigma^\prime$. As the stabilizer of the $3$-cell is trivial, there is a unique element $g\in G$
which takes $\sigma$ to $\sigma^\prime$. Let $\tau^\prime$ denote the pre-image $g^{-1}(\tau)$, a $2$-cell in the 
boundary of $\sigma$. Clearly $g$ identifies together the cells $\tau, \tau^\prime$ in the quotient space $X/G$.

If $\tau=\tau^\prime$, then the cell $\tau$ descends to a boundary cell in quotient space $X/G$, and the
stabilizer of $\tau$ is isomorphic to $\mathbb Z_2$ (with non-trivial element given by $g$). On the other hand, if $\tau\neq
\tau^\prime$, then $\tau$ descends to an interior cell in the quotient space $X/G$, with trivial stabilizer. 

Now if the $3$-cell $\sigma$ has a boundary $2$-cell $\tau$ whose stabilizer is $\mathbb Z_2$, then the orbit of $\tau$
intersects the boundary of $\sigma$ in precisely $\tau$. Looking in the coordinate corresponding to the orbit of $\tau$,
we see that in this case the map 
$\mathbb Z \longrightarrow \mathcal  \bigoplus _{f\in (X/G)^{(2)}} K_q\big(C^*_r(G_{f})\big)$ in
the chain complex is an {\it injection}, and hence that $E^2_{3,q}= H_3(\mathcal C) = 0$ for all even $q$. 

The other possibility is that {\it all} boundary $2$-cells are pairwise identified, in which case the quotient space 
$X/G$ is (topologically) a closed manifold. With respect to the induced orientation on the boundary of $\sigma$,
if any boundary $2$-cell $\tau$ is identified by an orientation {\it preserving} pairing to $\tau^\prime$, then the quotient space 
$X/G$ is a non-orientable manifold. Focusing on the coordinate corresponding to the orbit of $\tau$, we 
again see that the map $\mathbb Z \longrightarrow \mathcal  \bigoplus _{f\in (X/G)^{(2)}} K_q\big(C^*_r(G_{f})\big)$ in
the chain complex is injective (the generator of $\mathbb Z$ maps to $\pm 2$ in the $\tau$-coordinate). So in this case
we again conclude that $E^2_{3,q}= H_3(\mathcal C) = 0$ for all even $q$. 

Finally, we have the case where all pairs of boundary $2$-cells are identified together using orientation reversing pairings.
Then the quotient space $X/G$ is (topologically) a closed {\it orientable} manifold. In this case, the corresponding map  
$\mathbb Z \longrightarrow \mathcal  \bigoplus _{f\in X^{(2)}} K_q\big(C^*_r(G_{f})\big)$ in
the chain complex is just the zero map (the generator of $\mathbb Z$ maps to $0$ in each $\tau$-coordinate, due to the
two occurrences with opposite orientations). We summarize our 
discussion in the following:

\begin{lemma}\label{third-homology}
For our groups $G$, the third homology group $H_3(\mathcal C)$ is either (i)
isomorphic to $\mathbb Z$, if the quotient space $X/G$ is topologically a closed orientable manifold, or
(ii) trivial in all remaining cases. 
\end{lemma}

\vskip 5pt

\noindent {\bf Remark:} In \cite[Lemma 3.21]{MV}, it is shown that the comparison map from 
$H_i(\mathcal C)$ to the ordinary homology of the quotient space $H_i(\underline{B}G; \mathbb Z)$ is an isomorphism 
in all degrees $i> \dim (\underline{E}G^{\text{sing}}) +1$, and injective in degree $i= \dim (\underline{E}G^{\text{sing}}) +1$. 
Note that most of our Lemma \ref{third-homology} can also be deduced from this result. Indeed, our discussion shows that, 
in case (i), the singular set is $1$-dimensional (i.e.~all cells of dimension $\geq 2$ have trivial stabilizer), and hence 
$H_3(\mathcal C) \cong H_3(X/G) \cong \mathbb Z$. If $X/G$ is non-orientable, then \cite[Lemma 3.21]{MV} gives that 
$H_3(\mathcal C)$ injects into $H_3(X/G) \cong \mathbb Z_2$, so our Lemma provides a bit more information. In the 
case where $X/G$ has boundary, \cite[Lemma 3.21]{MV} implies that $H_3(\mathcal C)$ injects into $H_3(X/G) \cong 0$, 
so again recovers our result. We chose to retain our original proof of Lemma \ref{third-homology}, as a very similar argument
will be subsequently used to calculate $H_2(\mathcal C)$ (which does not follow from \cite[Lemma 3.21]{MV}).

\vskip 10pt


\subsection{$2$-skeleton of $X/G$ and the rank of $H_2(\mathcal C)$}\label{h2:section}

Now we turn our attention to the group $H_2(\mathcal C)$. In order to describe this homology group, we will
continue the analysis initiated in the previous section. Recall that we have an 
explicit (combinatorial) polyhedron $P$ which serves as a fundamental 
domain for the $G$-action. We can view the quotient space $X/G$ as obtained from the polyhedron
$P$ by identifying together certain faces of $P$.
The CW-structure on $X/G$ is induced from the natural (combinatorial) CW-structure on the polyhedron
$P$. The quotient space $X/G$ inherits the structure of a $3$-dimensional orbifold.
Note that, if we forget the orbifold structure and just think about the underlying topological space, then
 $X/G$ is a compact manifold, with possibly non-empty boundary. 
 
There is a close relationship between the {\it isotropy} of the cells 
in $X/G$, thought of as a 3-orbifold,
and the {\it topology} of $X/G$, viewed as a topological manifold. Indeed, as was discussed in the 
previous Section \ref{h3:section}, the stabilizer of any face $\sigma$ of the polyhedron $P$
is either (i) trivial, or (ii) is isomorphic to $\mathbb Z_2$. In the first case, there is an element in $G$ which
identifies the face $\sigma$ with some other face of $P$. So at the level of the quotient space $X/G$,
$\sigma$ maps to a $2$-cell which lies in the {\it interior} of the closed manifold $X/G$. In the second 
case, there are no other faces of the polyhedron $P$ that lie in the $G$-orbit of $\sigma$, and hence $\sigma$ 
maps to a boundary $2$-cell of $X/G$. We summarize this analysis in the following

\vskip 5pt

\noindent {\bf Fact 2:} For any $2$-cell $\sigma$ in $X/G$, we have that:
\begin{enumerate}[i)]
\item $\sigma$ lies in the boundary of $X/G$ if and only if $\sigma$ has isotropy $\mathbb Z_2$, and
\item $\sigma$ lies in the interior of $X/G$ if and only if $\sigma$ has trivial isotropy.
\end{enumerate}

\vskip 5pt

A similar analysis applies to $1$-cells. Indeed, 
the stabilizer of any edge in the polyhedron $P$ must either be (i) a finite cyclic group, or (ii) a finite dihedral
group. But case (ii) can only occur if there is some orientation reversing isometry through one of the faces containing
the edge. This would force the edge to lie in the boundary of the corresponding face, with the
stabilizer of the face being $\mathbb Z_2$. In view of {\bf Fact 2}, such an edge would have to lie in the
boundary of $X/G$. Conversely, if one has an edge in the boundary of $X/G$, then it has two adjacent faces
(which might actually coincide) in the boundary of $X/G$, each with stabilizer $\mathbb Z_2$, given by a 
reflection in the face. In most cases, these two reflections will determine a dihedral stabilizer for $e$; the
exception occurs if the two incident faces have stabilizers which coincide in $G$. In that case, the stabilizer
of $e$ will also be a $\mathbb Z_2$, and will coincide with the stabilizers of the two incident faces. We
summarize this discussion as our:

\vskip 5pt

\noindent {\bf Fact 3:} For any $1$-cell $e$ in $X/G$, we have that:
\begin{enumerate}[i)]
\item $e$ lies in the interior of $X/G$ if and only if $e$ has isotropy a cyclic group, acting by rotations around
the edge,
\item if $e$ has isotropy a dihedral group, then $e$ lies in the boundary of $X/G$,
\item the remaining edges in the boundary of $X/G$ have stabilizer $\mathbb Z_2$, which coincides
with the $\mathbb Z_2$ stabilizer of the incident boundary faces.
\end{enumerate}

\vskip 5pt

With these observations in hand, we are now ready to calculate $H_2(\mathcal C)\otimes \mathbb Q$. 
In order to understand this group, we need to understand the kernel of the morphism:
$$\Phi: \bigoplus _{\sigma\in (X/G)^{(2)}} K_0\big(C^*_r(G_{\sigma})\big)
\rightarrow \bigoplus _{e\in (X/G)^{(1)}} K_0\big(C^*_r(G_{e})\big).$$
Indeed, the group $H_2(\mathcal C)$ is isomorphic to the quotient of $\ker (\Phi)$ by a homomorphic 
image of $K_0\big(C^*_r(G_{\tau})\big) \cong \mathbb Z$, where $\tau$ is a representative for the unique
$3$-cell orbit. As such, we see that the rank of $H_2(\mathcal C)\otimes \mathbb Q$ either coincides with
the rank of $\ker(\Phi)$, or is one less than the rank of $\ker (\Phi)$.

Our approach to analyzing $\ker(\Phi)$ is to split up this group into smaller pieces, which are more
amenable to a geometric analysis. Let us introduce the notation $\Phi _e$, where $e$ is an edge, 
for the composition of the map $\Phi$ with the projection onto the summand $K_0\big(C^*_r(G_{e})\big)$.
The next Lemma analyzes the behavior of the map $\Phi$ in the vicinity of a boundary edge with
stabilizer a dihedral group.

\begin{lemma}\label{bdry-dihedral-edge}
Let $e$ be a boundary edge, with stabilizer a dihedral group $D_n$. Then we have:
\begin{enumerate}[(i)]
\item if $\sigma$ is an incident interior face, then 
$$\Phi_e\Big( K_0\big(C^*_r(G_{\sigma})\big)\Big)\subseteq
\mathbb Z \cdot \langle 1, 1, \ldots , 1, 1\rangle \leq K_0\big(C^*_r(D_n)\big),$$

\item if $\sigma_1, \sigma_2$ are the incident boundary faces, then
$$\Phi_e\Big( K_0\big(C^*_r(G_{\sigma_1})\big) \oplus K_0\big(C^*_r(G_{\sigma_2})\big)\Big) 
\cap \mathbb Z \cdot \langle 1, 1, \ldots , 1, 1\rangle = \langle 0 , \ldots ,0\rangle.$$
\end{enumerate}
\end{lemma}

Note that Lemma \ref{bdry-dihedral-edge} tells us that, from the viewpoint of finding elements
in $\ker (\Phi)$, boundary faces and interior faces that come together along an
edge with dihedral stabilizer have {\it no interactions.} 

\begin{proof}
There are precisely {\it two} boundary faces 
which are incident to $e$, and some indeterminate number of interior faces which are incident to $e$.
From {\bf Fact 2}, the boundary faces each have corresponding $G_\sigma \cong \mathbb Z_2$,
while the interior faces each have $G_\sigma \cong 1$. For the boundary faces, we have 
$$K_0\big(C^*_r(G_{\sigma})\big) = K_0\big(C^*_r(\mathbb Z_2)\big) \cong \mathbb Z\oplus \mathbb Z$$
with generators given by the trivial representation and the sign representation of the group $\mathbb Z_2$.
The interior faces have $K_0\big(C^*_r(G_{\sigma})\big) \cong \mathbb Z$, generated by the trivial 
representation of the trivial group. 

For each incidence of $\sigma$ on $e$, the effect of $\Phi_e$
on the generator is obtained by inducing up representations. But the trivial representation of the trivial group
always induces up to the left regular representation on the ambient group. The latter is the sum of all irreducible
representations, hence corresponds to the element $\langle 1, \ldots , 1 \rangle \leq K_0\big(C^*_r(\mathbb Z_n)\big)$.
This tells us that, for each internal face, the image of $\Phi_e$ lies in the subgroup 
$\mathbb Z \cdot \langle 1, 1, \ldots , 1, 1\rangle$, establishing (i).

On the other hand, an easy calculation (see Appendix A) shows that, if $\sigma_1, \sigma_2 \in X_1$ 
are the two boundary faces incident to $e$, then in the $e$-coordinate we have 
$$\Phi_e\Big( K_0\big(C^*_r(G_{\sigma_1})\big) \oplus K_0\big(C^*_r(G_{\sigma_2})\big)\Big) 
\cap \mathbb Z \cdot \langle 1, 1, \ldots , 1, 1\rangle = \langle 0 , \ldots ,0\rangle,$$
which is the statement of (ii).
\end{proof}

To analyze $\ker(\Phi)$, we need to introduce some auxiliary spaces. Recall that $X/G$
is topologically a closed $3$-manifold, possibly with boundary. 
We introduce the following terminology for boundary components:
\begin{itemize}
\item a boundary component is {\it dihedral} if it has no edges with stabilizer $\mathbb Z_2$
(i.e.~all its edges have stabilizers which are dihedral groups),
\item a boundary component is {\it non-dihedral} if it is not dihedral 
(i.e.~it contains at least one edge with stabilizer $\mathbb Z_2$),
\item a boundary component is {\it even} if it contains an edge $e$ with stabilizer
of the form $D_{2k}$ (i.e.~an edge whose stabilizer has order a multiple of $4$), and
\item a boundary component is {\it odd} if it is not even (i.e.~all its edges have stabilizers of the form 
$D_{2k+1}$).
\end{itemize}
Let $s$ denote the number of orientable even dihedral boundary components, and let $t$ denote
the number of orientable odd dihedral boundary components. 
Note that it is straightforward to calculate the integers $s,t$ from the
polyhedral fundamental domain $P$ for the $G$-action on $X$.

Next, form the $2$-complex $Y$ by taking the union of the closure of all interior faces of $X/G$, along with 
all the non-dihedral boundary components. We denote by $\partial Y \subset Y$ the subcomplex consisting
of all non-dihedral boundary components. By construction, $\partial Y$ consists precisely of the subcomplex
generated by the $2$-cells in $Y \cap \partial (X/G)$, so the choice of notation should cause no confusion. 
Let $Z$ denote the union of all dihedral boundary components of $X/G$. 

By construction, every $2$-cell in $X/G$ appears either in $Y$ or in $Z$, but not in both.
This gives rise to a decomposition of the 
indexing set $(X/G)^{(2)}= Y^{(2)}\coprod Z^{(2)}$, which in turn yields a splitting:
$$\bigoplus _{\sigma\in (X/G)^{(2)}} K_0\big(C^*_r(G_{\sigma})\big) =
\Big[\bigoplus _{\sigma\in Y^{(2)}} K_0\big(C^*_r(G_{\sigma})\big)\Big] \oplus 
\Big[\bigoplus _{\sigma\in Z^{(2)}} K_0\big(C^*_r(G_{\sigma})\big)\Big].$$
Let us denote by $\Phi_Y, \Phi_Z$ the restriction of $\Phi$ to the first and second
summand described above. We then have the following:

\begin{lemma}\label{split:lemma}
There is a splitting $\ker(\Phi) = \ker(\Phi_Y) \oplus \ker(\Phi_Z)$.
\end{lemma}

\begin{proof}
We clearly have the inclusion $\ker(\Phi_1) \oplus \ker(\Phi_2) \subseteq \ker(\Phi)$, so let us focus on the
opposite containment. If we have some arbitrary element $v\in \ker(\Phi)$, we can decompose $v=v_Y + v_Z$, where 
we have $v_Y \in \bigoplus _{\sigma\in Y^{(2)}} K_0\big(C^*_r(G_{\sigma})\big)$, and
$v_Z \in \bigoplus _{\sigma\in Z^{(2)}} K_0\big(C^*_r(G_{\sigma})\big)$. Let us first argue that $v_Z\in \ker (\Phi_Z)$,
i.e.~that $\Phi(v_Z)=0$. This is of course equivalent to showing that for every edge $e$, we have $\Phi_e(v_Z)=0$.

Since $v_Z$ is supported on $2$-cells lying in $Z$, 
it is clear that for any edge $e\not \subset Z$, we have $\Phi_e(v_Z)=0$. For edges $e\subset Z$, we have:
$$0= \Phi_e(v) = \Phi_e(v_Y+v_Z) = \Phi_e(v_Y) + \Phi_e(v_Z).$$
This tells us that $\Phi_e(v_Z) = \Phi_e( - v_Y)$ lies in the intersection
\begin{equation}\label{intersect}
 \Phi_e\Big( \bigoplus _{\sigma\in Y^{(2)}} K_0\big(C^*_r(G_{\sigma})\big)\Big) \cap 
 \Phi_e\Big( \bigoplus _{\sigma\in Z^{(2)}} K_0\big(C^*_r(G_{\sigma})\big)\Big).
\end{equation}
But $Y^{(2)}$ contains all the {\it interior} faces incident to $e$, while $Z^{(2)}$ contains all 
{\it boundary} faces incident to $e$. Since $e \subset Z$, the union of all dihedral boundary components of $X/G$, 
we have that the stabilizer $G_e$ must be dihedral. Applying
Lemma \ref{bdry-dihedral-edge}, we see that the intersection in equation (\ref{intersect}) consists of just the zero vector, 
and hence $\Phi_e(v_Z) =0$. 

Since we have shown that $\Phi_e(v_Z) =0$ holds for all edges $e$, we obtain that 
$v_Z\in \ker (\Phi_Z)$, as desired. Finally, we have that 
$$\Phi(v_Y) = \Phi( v - v_Z) = \Phi(v) - \Phi(v_Z) = 0$$
as both $v, v_Z$ are in the kernel of $\Phi$. We conclude that $v_Y \in \ker (\Phi_Y)$, concluding the proof of the Lemma. 
\end{proof}

We now proceed to analyze each of $\ker(\Phi_Y), \ker(\Phi_Z)$ separately. We start with:

\begin{lemma}\label{Z-term:lemma}
The group $\ker(\Phi_Z)$ is free abelian, of rank equal to $s+2t$.
\end{lemma}

Before establishing Lemma \ref{Z-term:lemma}, recall that $s,t$ counts the number of orientable dihedral boundary 
components of $X/G$ which are even and odd, respectively.  From the definition of $Z$, we see that the number of 
connected components of the space $Z$ is precisely $s+t$. 

\begin{proof}
It is obvious that $\ker(\Phi_Z)$ decomposes as a direct sum of the kernels of $\Phi$ restricted to the
individual connected components of $Z$, which are precisely the dihedral boundary components of $X/G$.
So we can argue one dihedral boundary component at a time. On a fixed 
dihedral boundary component, we have that each $2$-cell contributes a $\mathbb Z \oplus \mathbb Z$ to the source of the 
map $\Phi$, with canonical (ordered) basis given by the trivial representation and the sign representation on
$\mathbb Z_2$. Fix a boundary edge $e$, and let $\sigma_1, \sigma_2$ be the two boundary
faces incident to $e$. We assume that the two faces are equipped with compatible orientations,
and let $(a_i, b_i)$ be elements in the groups $K_0\big(C^*_r(G_{\sigma_i})\big) \cong \mathbb Z \oplus \mathbb Z$.
Now assume that $\Phi_e\big(( a_1, b_1 \ | \ a_2, b_2 ) \big) =0$. Then an easy computation (see Appendix A) shows that:

\vskip 5pt

\begin{enumerate}[a)]
\item if $e$ has stabilizer of the form $D_{2k+1}$, then we must have $a_1=a_2$ and $b_1=b_2$,
\item if $e$ has stabilizer of the form $D_{2k}$, then we must have $a_1=a_2=b_1=b_2$
\end{enumerate}

\vskip 5pt

\noindent (and since $Z$ consists of dihedral boundary components, there are no edges $e$ in $Z$ with stabilizer 
$\mathbb Z_2$).
Note that reversing the orientation on one of the faces just changes the sign of the corresponding entries.
We can now calculate the contribution of each boundary component to $\ker(\Phi_Z)$. 

\vskip 5pt

\noindent \underline{Non-orientable components:} Any such boundary component contains an embedded 
M\"obius band. Without loss of generality, we can assume that the sequence of faces $\sigma_1, \ldots , \sigma _r$
cyclically encountered by this M\"obius band are all distinct. At the cost of flipping the orientations on $\sigma_i$,
$2\leq i\leq r$, we can assume that consecutive pairs are coherently oriented. Since we have a M\"obius band,
this forces the orientations of $\sigma_1$ and $\sigma _r$ to be non-coherent along their common edge. 
So if we have an element lying in $\ker(\Phi _Z)$, the coefficients along the cyclic sequence of faces must satisfy
(regardless of the edge stabilizers):
$$a_1 = a_2 = \ldots = a_k = - a_1$$
$$b_1= b_2 = \ldots = b_k = -b_1$$
This forces $a_1=b_1=0$. Regardless of the orientations and edge stabilizers, equations (a) and (b) imply
that this propagates to force
all coefficients to equal zero. We conclude that any element in $\ker (\Phi _Z)$ must have all zero 
coefficients in the $2$-cells corresponding to any non-orientable boundary component.

\vskip 5pt

\noindent \underline{Orientable odd components:} Fix a coherent orientation of all the $2$-cells in the 
boundary component. Then in view of equation (a) above, elements lying in $\ker(\Phi)$ must have
all $a_i$-coordinates equal, and all $b_i$-coordinates equal (as one ranges over $2$-cells within this
fixed boundary component). This gives two degrees of freedom, and hence such a boundary component
contributes a $\mathbb Z^2$ to $\ker(\Phi_Z)$.

\vskip 5pt

\noindent \underline{Orientable even components:} Again, let us fix a coherent orientation of all the $2$-cells in the 
boundary component. As in the odd component case, any element in $\ker(\Phi_Z)$ must have
all $a_i$-coordinates equal, and all $b_i$-coordinates equal. However, the presence of a single edge
with stabilizer of the form $D_{2k}$ forces, for the two adjacent faces, to have corresponding 
$a$- and $b$-coordinates equal (see equation (b) above). This in turn propagates to yield that {\it all} 
the $a$- and $b$-coordinates
must be equal. As such, we have one degree of freedom for elements in the kernel, and hence such a 
boundary component contributes a single $\mathbb Z$ to $\ker (\Phi_Z)$. This concludes the proof of Lemma
\ref{Z-term:lemma}.
\end{proof}

Next we focus on the group $\ker(\Phi_Y)$. We would like to relate $\ker (\Phi _Y)$ with the second homology of 
the space $Y$. Let 
$\mathcal A$ denote the cellular chain complex for the CW-complex $Y$, and let 
$d_Y: \mathcal A_2 \rightarrow \mathcal A_1$ denote the differentials in the cellular 
chain complex. Since $Y$ is a $2$-dimensional CW-complex, 
we have that $H_2(Y) = \ker (d_Y)$. Our next step is to establish:

\begin{lemma}\label{Y-term:lemma}

There is a split surjection $\phi: \ker(\Phi_Y) \rightarrow \ker(d_Y)$, providing a direct sum decomposition
$\ker(\Phi_Y) \cong \ker(\phi) \oplus \ker(d_Y)$.
\end{lemma}

\begin{proof}
Let $\mathcal D \subset \mathcal C$ denote the subcomplex
of our original chain complex determined by the subcollection of indices $Y^{(k)} \subset (X/G)^{(k)}$.
By construction, the map $\Phi _Y$ we are interested in is the boundary operator $\Phi_Y: \mathcal D_2\rightarrow
\mathcal D_1$ appearing in the chain complex $\mathcal D$. We define the map 
$$\hat \phi: \mathcal D_2=  \bigoplus _{\sigma\in Y^{(2)}} K_0\big(C^*_r(G_{\sigma})\big) \rightarrow 
\bigoplus _{\sigma\in Y^{(2)}} \mathbb Z = \mathcal A_2$$ 
as the direct sum of maps $\hat \phi_\sigma: K_0\big(C^*_r(G_{\sigma})\big) \rightarrow \mathbb Z$, where:
\begin{itemize}
\item if $G_\sigma$ is trivial, then $\hat \phi_\sigma: \mathbb Z \rightarrow \mathbb Z$ takes the generator for
$K_0\big(C^*_r(G_{\sigma})\big) = \mathbb Z$ given by the trivial representation to the element $1\in \mathbb Z$, and
\item if $G_\sigma = \mathbb Z_2$, then $\hat \phi_\sigma: \mathbb Z \oplus \mathbb Z \rightarrow \mathbb Z$ is given
by $\hat \phi_\sigma(\langle 1, 0 \rangle) = 1, \hat \phi_\sigma(\langle 0, 1\rangle) = 0$, where, as usual, $\langle 1, 0 \rangle, 
\langle 0, 1 \rangle$ correspond to the trivial representation and the sign representation respectively.
\end{itemize}
For any element $z \in \ker (\Phi_Y)$, a computation shows that $(d_Y\circ \hat \phi)(z) = 0$, and hence $\hat \phi$
restricts to a morphism $\phi: \ker (\Phi _Y) \rightarrow \ker (d_Y)$.

Next, we argue that the map $\phi: \ker (\Phi _Y) \rightarrow \ker (d_Y)$ is surjective. To see this, we construct 
a map $\bar \phi: \mathcal A_2 \rightarrow \mathcal D_2$ as a direct sum of maps $\bar \phi_\sigma: \mathbb Z\rightarrow
K_0\big(C^*_r(G_{\sigma})\big)$. In terms of our usual generating sets for the groups $K_0\big(C^*_r(G_{\sigma})\big)$,
the maps $\bar \phi_\sigma$ are given by:
\begin{itemize}
\item if $G_\sigma$ is trivial, then $\bar \phi_\sigma: \mathbb Z \rightarrow \mathbb Z$ is defined by $\bar \phi_\sigma(1)=1$, and
\item if $G_\sigma = \mathbb Z_2$, then $\bar \phi_\sigma: \mathbb Z \rightarrow \mathbb Z \oplus \mathbb Z$ is defined by 
$\bar \phi_\sigma(1)=\langle 1, 1 \rangle$.
\end{itemize}
We clearly have that $\hat \phi \circ \bar \phi: \mathcal A_2 \rightarrow \mathcal A_2$ is the identity, and an easy 
computation shows that if $z\in \ker (d_Y)$, then $\bar \phi (z) \in \ker (\Phi_Y)$. We conclude that the restriction 
$\phi: \ker (\Phi _Y) \rightarrow \ker (d_Y)$ is surjective, and that the restriction of $\bar \phi$ to 
$\ker(d_Y)$ provides a splitting of this surjection. 
Since the map $\phi$ is a split surjection, we see that $\ker (\Phi_Y) \cong \ker (d_Y) \oplus \ker (\phi)$, 
completing the proof of Lemma \ref{Y-term:lemma}.
\end{proof}

So the last
step is to identify $\ker (\phi)$. Recall that $Y$ is a $2$-complex which contains, as a subcomplex, the union of all boundary 
components of $X/G$ which have an edge with stabilizer $\mathbb Z_2$. We denoted this subcomplex by 
$\partial Y \subset Y$. We can again call a connected component in $\partial Y$ {\it even} if it contains some edge 
with stabilizer of the form $D_{2k}$, and {\it odd} otherwise. Let $t^\prime$ denote the
number of orientable, odd connected components in $\partial Y$. Then 
we have:

\begin{lemma}\label{phi-term:lemma}
The group $\ker (\phi)$ is free abelian, of rank $=t^\prime$.
\end{lemma}

\begin{proof}
From the definition of $\phi$, it is easy to see what form an element in $\ker(\phi)$
must have: in terms of the splitting $\mathcal D_2=  \bigoplus _{\sigma\in Y^{(2)}} K_0\big(C^*_r(G_{\sigma})\big)$,
the element can only have non-zero terms in the coordinates corresponding to $2$-cells in $\partial Y$. 
Moreover, in the coordinates $\sigma \in (\partial Y)^{(2)}$, the entries in the corresponding 
$K_0\big(C^*_r(G_{\sigma})\big) \cong \mathbb Z \oplus \mathbb Z$ must lie in the subgroup 
$\mathbb Z \cdot \langle 0, 1\rangle$. Finally, the fact that the elements we are considering lie in $\ker (\Phi _Y)$
means that, at each edge $e\in (\partial Y)^{(1)}$, with incident edges $\sigma_1, \sigma_2$, we must have 
that the corresponding coefficients $\langle 0, b_1\rangle \in K_0\big(C^*_r(G_{\sigma_1})\big)$ and 
$\langle 0, b_2\rangle \in K_0\big(C^*_r(G_{\sigma_2})\big)$ sum up to zero, i.e.~that $b_1+b_2=0$. These
properties almost characterize elements in $\ker (\phi)$. Clearly, we can again analyze the situation one connected
component of $\partial Y$ at a time. As in the argument for Lemma \ref{Z-term:lemma}, there are cases to consider:

\vskip 5pt

\noindent \underline{Even component:} In the case where an element $z\in \ker (\phi)$ is supported entirely
on an even boundary component, there is one additional constraint. For the two faces $\sigma_1, \sigma_2$
incident to the edge with stabilizer $D_{2k}$, the fact that $z\in \ker(\Phi)$ forces the corresponding coefficients to satisfy
$b_1=b_2=a_1=a_2$ (see equation (b) in the proof of Lemma \ref{Z-term:lemma}). Since $z\in \ker(\phi)$, we
also have $a_1=a_2=0$. This implies that the coefficients $b_1=b_2$ must also vanish. But then all the $b_i$
coefficients must vanish. We conclude that any element $z\in \ker(\phi)$ must have zero coefficients on all 
$2$-cells contained in an even component.

\vskip 5pt

\noindent \underline {Odd component:} In the case where an element $z\in \ker (\phi)$ is supported entirely
on an odd boundary component, the conditions discussed above actually do characterize an element 
in $\ker(\phi)$. This is due to the fact that, at every edge, the $b_i$ components are actually independent
of the $a_i$ components (see equation (a) in the proof of Lemma \ref{Z-term:lemma}). But the description
given above is just stating that the $b_i$ form the coefficients for an (ordinary) $2$-cycle in the boundary
component. Such a $2$-cycle can only exist if the boundary component is orientable, in which case there
is a $1$-dimensional family of such $2$-cycles. We conclude that the orientable, odd components each
contribute a $\mathbb Z$ to $\ker(\phi)$, while the non-orientable odd components make no contributions.

\vskip 5pt

\noindent{}Since $t^\prime$ is the number of orientable, odd components in $\partial Y$, the Lemma follows. 
\end{proof}

We now have all the required ingredients to establish:

\begin{Thm}
The group $\ker(\Phi)$ is free abelian of rank $s + t^\prime + 2t + \beta_2(Y)$.
\end{Thm}

\begin{proof}
Lemma \ref{split:lemma} provides us with a splitting $\ker(\Phi) = \ker(\Phi_Y) \oplus \ker(\Phi_Z)$. 
Lemma \ref{Z-term:lemma} shows that $\ker(\Phi_Z)$ is free abelian of rank $=s+2t$. Lemma
\ref{Y-term:lemma} yields the splitting $\ker(\Phi_Y) \cong \ker(\phi) \oplus \ker(d_Y)$. Finally, 
Lemma \ref{phi-term:lemma} tells us that $\ker(\phi)$ is free abelian of rank $=t^\prime$, while
the fact that $Y$ is a $2$-complex tells us that $\ker(d_Y)$ is free abelian of rank $=\beta_2(Y)$. 
\end{proof}

As a consequence, we obtain the desired formula for $\beta_2(\mathcal C)$.

\begin{corollary}\label{h2-term:thm}
For our groups $G$, we have that the rank of $H_2(\mathcal C)\otimes \mathbb Q$ is either:
\begin{itemize}
\item $\beta_2(Y)$ if $X/G$ is a closed, oriented, $3$-manifold, or
\item $s + t^\prime + 2t + \beta_2(Y)-1$ otherwise.
\end{itemize}
\end{corollary}

\vskip 10pt

\noindent{\bf Remark:} Corollary \ref{h2-term:thm} gives us an algorithmically efficient method for computing
$\beta_2(\mathcal C)$, as it merely requires counting certain boundary components of $X/G$ (to determine
the integers $s, t, t^\prime$), along with the calculation of the second Betti number of an explicit $2$-complex
(for the $\beta_2(Y)$ term).


\subsection{Euler characteristic and the rank of $H_1(\mathcal C)$}

Using the procedure described in the previous section, we will now assume that the ranks
$\beta_0(\mathcal C), \beta_2(\mathcal C)$, and $\beta_3(\mathcal C)$
have already been calculated. In order to compute the rank 
of $H_1(\mathcal C)\otimes \mathbb Q$, we recall that any chain complex has an associated {\it Euler characteristic}. 
The latter is defined to be the alternating sum of the ranks of the groups appearing in the chain complex.
It is an elementary exercise to verify that the Euler characteristic 
also coincides with the alternating sum of the ranks of the homology groups of the chain complex.

In our specific case, the Euler characteristic $\chi(\mathcal C)$ of the chain 
complex $\mathcal C$ 
can easily be calculated from the various groups $G_\sigma$, where $\sigma$ ranges over the cells in $\underline{B}G$.
Each cell $\sigma$ in $\underline{B}G$ contributes $(-1)^{\dim \sigma}c(G_\sigma)$, where $c(G_\sigma)$ is the 
number of conjugacy classes in the stabilizer $G_\sigma$ of the cell.
Since the homology groups $H_i(\mathcal C)$ vanish when $i\neq 0, 1, 2, 3$, we also have the alternate formula:
$$\chi (\mathcal C) = \beta_0(\mathcal C) - \beta_1(\mathcal C) + \beta_2(\mathcal C) - \beta_3(\mathcal C)$$
This allows us to solve for the rank of $H_1(\mathcal C)\otimes \mathbb Q$, yielding:

\begin{lemma}\label{h1-rank}
For our groups $G$, we have that the rank of $H_1(\mathcal C)\otimes \mathbb Q$ coincides with 
$\beta_1(\mathcal C) = \beta_0(\mathcal C) + \beta_2(\mathcal C) - \beta_3(\mathcal C) - \chi (\mathcal C)$.
\end{lemma}



\section{Some examples}

We illustrate our algorithm by computing the rational topological $K$-theory of several groups. 
The first two examples are classes of groups for which the topological $K$-theory has already been computed.
Since our algorithm does indeed recover (rationally) the same results, these examples serve as a 
check on our method. The last three examples provide some new computations.

The first example considers the particular case where $G$ is additionally assumed to be torsion-free. 
As a concrete special case, we deal with any semi-direct product of $\mathbb Z^2$ with $\mathbb Z$
(the integral computation for these groups can be found in the recent thesis of Isely \cite{I}). The second
example considers a finite extension of the integral Heisenberg group by $\mathbb Z_4$.
The {\it integral} topological $K$-theory (and algebraic $K$- and $L$- theory) for this group has already 
been computed by L\"uck \cite{Lu3}. 

The third and fourth classes of examples are hyperbolic Coxeter groups that have previously been considered 
by Lafont, Ortiz, and Magurn in \cite[Example 7]{LOM}, and \cite[Example 8]{LOM} respectively (where their lower 
algebraic $K$-theory was computed). The fifth example is an affine split crystallographic group, whose 
algebraic $K$-theory has been studied by Farley and Ortiz \cite{FO}.


\subsection{Torsion-free examples.}\label{torsion-free-examples}
 In the special case where $G$ is {\it torsion-free}, 
our algorithm becomes particularly simple, as we now proceed to explain.

\vskip 5pt

Let $G$ be a torsion-free group with a cocompact, $3$-manifold model $X$ for the classifying space $\underline{E}G=EG$.
Firstly, recall that $\beta _0 (\mathcal C)
= cf(G)$, where $cf(G)$ denotes the number of conjugacy classes of elements of finite order in $G$ (our
Lemma \ref{rank-K0} provides a way of computing this integer from the $1$-skeleton of $X/G$). Since $G$
is torsion-free, we obtain that $\beta _0 (\mathcal C)= 1$. 

Next, we consider the orbit space $M:=X/G$. Recall that any boundary component in the $3$-manifold
$M$ gives $2$-cells with stabilizer $\mathbb Z_2$. Since $G$ is torsion-free, the orbit space 
$M$ has no boundary, hence is a {\it closed} $3$-manifold. Then
Lemma \ref{third-homology} tells us that 
\[
	\beta_3(\mathcal C) = 
\begin{cases}
	1 & \text{if $M$ orientable},\\
	0 & \text{if $M$ non-orientable}.
\end{cases}
\]

To compute $\beta_2(\mathcal C)$ we apply Corollary \ref{h2-term:thm}. The 2-simplex $Y$ is just the 2-skeleton of $M$ 
and, as $\partial M= \emptyset$, we obtain that
\[
	\beta_2(\mathcal C) = 
\begin{cases}
	\beta_2(Y) & \text{if $M$ orientable},\\
	\beta_2(Y)-1 & \text{if $M$ non-orientable}.
\end{cases}
\]
Note that the $2^\text{nd}$ Betti number of $Y=M^{(2)}$ can be attained from that of $M$, as follows. Since $M$ is 
obtained from $Y$ by attaching a single 3-cell, the Mayer-Vietoris exact sequence gives
\[
	\xymatrix{0 \ar[r] & H_3(M) \ar@{^{(}->}[r] & H_2(S^2) \ar[r]^-g & H_2(Y) \oplus H_2(\mathbb D^3) \ar@{>>}[r] & H_2(M) \ar[r] & 0 } 
\]
(Here $\mathbb D^3$ is the attaching 3-disk.) 
Recall that $H_2(S^2)\cong \mathbb Z$ and $H_2(\mathbb D^3)=0$. Hence if $M$ is orientable, 
$H_3(M) \cong \mathbb Z$, the image of the map $g$ is then torsion and tensoring with $\mathbb Q$ gives 
$\beta_2(Y)=\beta_2(M)$. If $M$ is non-orientable, $H_3(M) = 0$, the map $g$ is injective and we have 
$\beta_2(Y)-1=\beta_2(M)$.
Hence in all cases we actually obtain that $\beta_2(\mathcal C)  = \beta_2(M)$.

To compute $\beta_1(\mathcal C)$ we should find $\chi(\mathcal C)$. Since $G$ is torsion-free all the 
isotropy groups are trivial and thus $\chi(\mathcal C) = \chi(M)$. Since $M$ is a closed 3-manifold, $\chi(M)$ and therefore 
$\chi(\mathcal C)$ are zero. Finally, Lemma \ref{h1-rank} gives
\[
	\beta_1(\mathcal C) = \beta_0(\mathcal C)  + \beta_2(\mathcal C)  - \beta_3(\mathcal C)  - \chi(\mathcal C) = \beta_2(M) - \beta_3(\mathcal C) +1,
\]
which simplifies to two cases
\[
\beta_1(\mathcal C) = 
\begin{cases}
	\beta_2(M) & \text{if $M$ is orientable},\\
 	\beta_2(M)+1 & \text{if $M$ is not orientable}.
\end{cases}
\]
Finally applying Lemma \ref{rational-expression}, we deduce the:

\begin{corollary}\label{torsion-free-case}
Let $G$ be a torsion-free group, and $X$ be a cocompact $3$-manifold model for $\underline{E}G=EG$.
Assume that the quotient $3$-manifold $M= X/G$ is geometrizable (this is automatic, for instance, if $M$ is orientable). 
Then we have that
\[
	\rank\left( K_q(C_r^\ast(G)) \otimes \mathbb Q \right) = \beta_2(M)+1
\]
holds for all $q$.
\end{corollary}

\vskip 5pt

\noindent{\bf Remark:} The number above is the sum of the even-dimensional Betti numbers of $M$ (which coincides with the sum of the odd-dimensional Betti numbers of $M$, by Poincar\'e duality) --- compare this with the Remark after Lemma \ref{rational-expression}.

\vskip 5pt

\noindent{\bf Remark:} Note that for $G$ torsion-free, the dimension of the singular part is $-1$ and hence Lemma 3.21 in \cite{MV} gives $H_i(\mathcal C) \cong H_i(M)$ for $i > 0$ and an injection $H_0(\mathcal C) \hookrightarrow H_0(M)$. From this it follows that $\beta_i(\mathcal C) = \beta_i(M)$ for $i=1,2,3$ and $\beta_0(\mathcal C) = \beta_0(M)$ since $1 \le \beta_0(\mathcal C) \le \beta_0(M)=1$. This is shown above by direct application of our algorithm.

\vskip 5pt

\noindent {\bf Semi-direct product of $\mathbb Z^2$ and $\mathbb Z$.} For a concrete example of the torsion-free case, consider a semi-direct product $G_\alpha= \mathbb Z^2 \rtimes_\alpha \mathbb Z$, 
where $\alpha \in \aut(\mathbb Z^2) = GL_2(\mathbb Z)$. The automorphism $\alpha$ can be realized (at the level of the 
fundamental group) by an affine self diffeomorphism of the 2-torus $T^2 = S^1 \times S^1$, $f \colon T^2 \to T^2$. The 
mapping torus $M_f$ of the map $f$ yields a closed 3-manifold which is aspherical and satisfies $\pi_1(M_f) \cong G_\alpha$. 
Hence it is a model of $BG_\alpha$ and its universal cover a model of $EG_\alpha$. Since $G_\alpha$ is torsion-free 
(as it is the semi-direct product of torsion-free groups), these spaces are also models of $\underline{B}G_\alpha$ 
respectively $\underline{E}G_\alpha$. In particular, these examples fall under the purview of Corollary \ref{torsion-free-case}, 
telling us that $\rank\left( K_q(C_r^\ast(G_\alpha)) \otimes \mathbb Q \right) = \beta_2(M_f)+1$.
To complete the calculation, we just need to compute the $2^\text{nd}$ Betti number of the 3-manifold $M_f$. 
This follows from an straightforward application of the Leray-Serre spectral sequence. We have included the 
details in Appendix B and here we only quote the result
\[
\beta_2(M_f) = 
\begin{cases}
3 & \text{if } \alpha = \text{Id},\\
2 & \text{if } \det(\alpha)=1, \text{tr}(\alpha)=2, \alpha \neq \text{Id},\\
1 & \text{if } \det(\alpha)=1, \text{tr}(\alpha)\neq 2,\\
1 & \text{if } \det(\alpha)=-1, \text{tr}(\alpha)= 0,\\
0 & \text{if } \det(\alpha)=-1, \text{tr}(\alpha)\neq 0.\\
\end{cases}
\]
Adding 1 we obtain 
\[
K_q(C_r^\ast(G_\alpha)) \otimes \mathbb Q  \cong  
\begin{cases}
\mathbb Q^4 & \text{if } \alpha = \text{Id},\\
\mathbb Q^3 & \text{if } \det(\alpha)=1, \text{tr}(\alpha)=2, \alpha \neq \text{Id},\\
\mathbb Q^2 & \text{if } \det(\alpha)=1, \text{tr}(\alpha)\neq 2,\\
\mathbb Q^2 & \text{if } \det(\alpha)=-1, \text{tr}(\alpha)= 0,\\
\mathbb Q  & \text{if } \det(\alpha)=-1, \text{tr}(\alpha)\neq 0.\\
\end{cases}
\]
These results agree with the integral computations in Isely's thesis \cite[pp.~5-7]{I}, giving us a first
check on our method.


\subsection{Nilmanifold example}\label{Nil-mfld-example} In the previous section, we discussed examples where
the group was torsion-free, and hence the quotient space was a closed $3$-manifold. In this next example, we have
a group {\it with} torsion, but with quotient space again a closed $3$-manifold.

\vskip 5pt

The real Heisenberg group $\hei(\mathbb R)$ is the Lie group of upper unitriangular, $3 \times 3$ matrices with real entries.
It is naturally homeomorphic to $\mathbb R^3$. The integral Heisenberg group $\hei (\mathbb Z)$ is the discrete subgroup 
consisting of matrices whose entries are in $\mathbb Z$. There is an automorphism $\sigma \in \aut \big(\hei(\mathbb R)\big)$ 
of order $4$ given by:
$$ \sigma:
\left[ \begin{array}{ccc}
1 & x & z \\
0 & 1 & y\\
0 & 0 & 1 \end{array} \right] \mapsto 
\left[ \begin{array}{ccc}
1 & -y & z-xy \\
0 & 1 & x\\
0 & 0 & 1 \end{array} \right].
$$
This automorphism restricts to an automorphism of the discrete subgroup $\hei (\mathbb Z)$, allowing us to define the 
group $G:= \hei (\mathbb Z) \rtimes \mathbb Z_4$. An explicit presentation of the group $G$ is given by

$$G:= \Bigg\langle a, b, c, t   \hskip 5pt \Bigg|  \hskip 5pt
\parbox{2.3in}
{\centerline{$[a,c] = [b, c] =1, \hskip 5pt [a, b] =c, \hskip 5pt t^4=1$}

\centerline{$tat^{-1}=b, \hskip 5pt tbt^{-1}= a^{-1}, \hskip 5pt tct^{-1} = c$}}
\Bigg\rangle $$
where as usual, $[x,y]$ denotes the commutator of the elements $x, y$. In the above presentation, we are identifying the
generators $a, b, c$ with the matrices in $\hei (\mathbb Z)$ given by
$$T_a=
\left[ \begin{array}{ccc}
1 & 1 & 0 \\
0 & 1 & 0\\
0 & 0 & 1 \end{array} \right] ,
\hskip 5pt
T_b=
\left[ \begin{array}{ccc}
1 & 0 & 0 \\
0 & 1 & 1\\
0 & 0 & 1 \end{array} \right] ,
\hskip 5pt
T_c=
\left[ \begin{array}{ccc}
1 & 0 & 1 \\
0 & 1 & 0\\
0 & 0 & 1 \end{array} \right].
$$
These generate the normal subgroup $\hei (\mathbb Z ) \triangleleft G$, 
while the conjugation by the last generator $t$ acts via the automorphism $\sigma \in \aut \big( \hei (\mathbb Z)\big)$.

The action of $\hei (\mathbb Z)$ on $\hei (\mathbb R)$ given by left multiplication and the action of $\mathbb Z_4$ on 
$\hei (\mathbb R)$ given by the automorphism $\sigma$ fit together to give an action of the group $G$ on $\hei (\mathbb R)$.
It is shown in \cite[Lemma 2.4]{Lu3} that this action on $\hei (\mathbb R)$ provides a cocompact model for $\underline{E}G$,
with orbit space $G \backslash \underline{E}G$ homeomorphic to $S^3$. In order to apply our algorithm, we need to identify
a $G$-CW-structure on $\hei (\mathbb R)$. Let us identify $\mathbb R^3$ with $\hei (\mathbb R)$ via the map:
$$(x, y, z) \leftrightarrow
\left[ \begin{array}{ccc}
1 & x & z \\
0 & 1 & y\\
0 & 0 & 1 \end{array} \right].
$$
Via this identification, we will think of $G$ as acting on $\mathbb R^3$. 

The action of the index four subgroup $\hei (\mathbb Z) \triangleleft G$ on $\mathbb R^3$
\[
    (n,m,l) \cdot (x,y,z) = (x+n,y+m,z+ny+l)
\]
 is free. The quotient
space $\hei(\mathbb Z) \backslash \mathbb R^3$ can be identified in two steps. First, we quotient out by the 
normal subgroup $H:=\langle T_b, T_c\rangle \cong \mathbb Z \otimes \mathbb Z$. On any hyperplane given by fixing the $x$-coordinate $x = x_0$, the subgroup $H$ leaves the hyperplane invariant, with the generators $T_b, T_c$ translating by one in the $y$ and $z$ coordinates respectively. Quotienting out by $H$, we obtain that $H \backslash \mathbb R^3$ is homeomorphic to $\mathbb R \times T^2$, where the $T^2$ refers to the standard torus obtained from the unit square (centered at the origin) by identifying the opposite sides. The quotient $\hei (\mathbb Z) \backslash \mathbb R^3$ can now be identified by looking at the action of the quotient group $\hei(\mathbb Z) \slash H$ on the space $\mathbb R \times T^2$. The
generator for $\mathbb Z\cong \hei(\mathbb Z) \slash H$, being the image of the matrix $T_x\in  \hei(\mathbb Z)$, acts by 
$(x, y, z) \mapsto (x+1, y, z+y)$. Putting this together, we see that a fundamental domain for the $\hei (\mathbb Z)$-action
on $\mathbb R^3$ is given by the unit cube $[-1/2, 1/2] ^3$ centered at the origin. The quotient $3$-manifold 
$M:=\hei(\mathbb Z) \backslash \mathbb R^3$ can now be obtained from the cube via a suitable identification of the faces. 
The manifold $M$ can also be thought of as the mapping torus of the map $\phi: T^2\rightarrow T^2$ given by $(y, z)\mapsto
(y, y+z)$ (mod $1$).

Next, we identify a fundamental domain for the $G$-action on $\mathbb R^3$. Observe that, since 
$\hei (\mathbb Z) \triangleleft G$, there is an induced $G/\hei (\mathbb Z) \cong \mathbb Z_4$ on 
$M$, and a natural identification between $G\backslash \mathbb R^3$
and $\mathbb Z_4 \backslash M$. The manifold $M$ naturally fibers over $T^2$, with fiber $S^1$, via the projection
onto the $(x, y)$-plane. The $\mathbb Z_4$ action preserves the $S^1$-fibers, so induces an action on the $2$-torus
$T^2$. At the level of the fundamental domain $[-1/2, 1/2]^2 \subset \mathbb R^2$ in the $(x,y)$-plane, 
the $\mathbb Z_4$-action is given
by $(x, y) \mapsto (-y, x)$. This tells us that a fundamental domain for the $\mathbb Z_4$-action can be 
obtained by restricting to the square $[0,1/2] \times [0,1/2]$. As far as the isotropy goes, 
there are four points in $T^2$ with non-trivial
stabilizer: the images of points $(0,0)$ and $(1/2, 1/2)$ both have stabilizer $\mathbb Z_4$, and the images of
the points $(0, 1/2)$ and $(1/2, 0)$, both have stabilizer $\mathbb Z_2$ (and lie in the same $\sigma$-orbit).  

We conclude that a fundamental domain for the $G$-action on $\mathbb R^3$ is given by 
the rectangular prism $P:=[0, 1/2] \times [0, 1/2] \times [-1/2, 1/2] \subset \mathbb R^3$ (Figure 1). The interior of $P$ gives the
single $3$-cell orbit for the equivariant polyhedral $G$-CW-structure on $\mathbb R^3$. 
For the isotropy groups, we just need to understand the action on the four vertical lines 
lying above each of the four points $(0,0), (1/2, 0), (0, 1/2),$ and $(1/2, 1/2)$. It is easy to see that the
vertical line $(0,0,z)$ consists entirely of points with stabilizer $\mathbb Z_4$, while the vertical lines 
$(1/2, 0, z)$ and $(0, 1/2, z)$ both have stabilizer $\mathbb Z_2$. On the other hand, the action of the 
element of order $4$ on the $S^1$-fiber above the point $(1/2, 1/2)$ can be calculated, 
and consists of a rotation by $\pi/4$ on the $S^1$-fiber. So the stabilizers for points on the line $(1/2, 1/2, z)$ 
are all trivial.


\begin{figure}
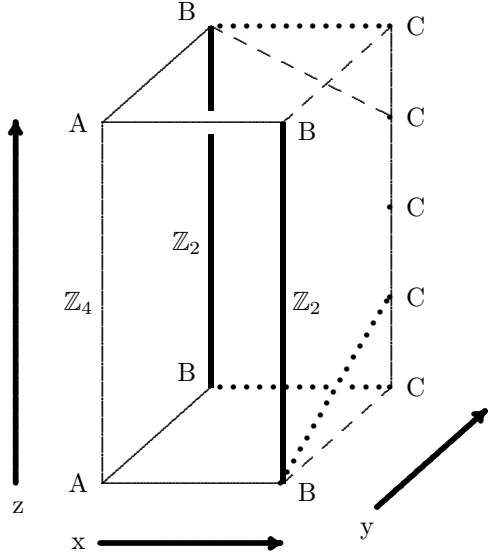
   
\label{fig-1}
\begin{center}

\vbox{\beginpicture
  \setcoordinatesystem units <1.6cm,1.6cm> point at -.4 2.5
  \setplotarea x from -.8 to 3, y from -.6 to 4

  \def\smallbul{\hskip .8pt\circle*{2.2}}
  \linethickness=.7pt

 \def\myarrow{\arrow <4pt> [.2, 1]}


\setsolid
  \plot 1.7 0 3.2 0   /
  \plot 1.7 3 3.2 3 /
  \plot 1.7 0 2.6 0.8 /
  \plot 1.7 3 2.6 3.8 /
  \plot 1.7 0 1.7 3 /
  \plot 4.1 0.8 4.1 3.8 /

\linethickness=2pt
  \putrule from 3.2 0 to 3.2 3  

\linethickness=2pt
  \putrule from 2.6 0.8 to 2.6 2.9

\linethickness=2pt
  \putrule from 2.6 3.1 to 2.6 3.8

\setdashes
  \plot 3.2 0 4.1 0.8   /
  \plot 3.2 3 4.1 3.8   /
  \plot 2.6 3.8 4.1 3.05 /

\setdots
\setplotsymbol ({\circle*{2}}) 
  \plot 2.7 3.8 4.1 3.8   /
  \plot 2.7 0.8 3.6 0.8   /
  \plot 3.8 0.8 4.1 0.8 /
  \plot 3.2 0 4.1 1.55 /

  \put {\smallbul} at 4.1 0.8 
  \put {\smallbul} at 4.1 3.8
  \put {\smallbul} at 4.1 2.3
  \put {\smallbul} at 4.1 1.55
  \put {\smallbul} at 4.1 3.05
  
  \put {A}  at  1.5 0
  \put {A}  at  1.5 3
  \put {$\mathbb Z_4$} at 1.5 1.5
  \put {B} [t] at  3.4 0
  \put {B} [t] at  3.4 3 
  \put {$\mathbb Z_2$} at 3.4 1.5                                                   
  \put {B} [t] at  2.4 1
  \put {B} [t] at  2.4 4
  \put {$\mathbb Z_2$} at 2.4 2                                                
  \put {C} at 4.3 0.8 
  \put {C} at 4.3 3.8
  \put {C} at 4.3 2.3
  \put {C} at 4.3 1.55
  \put {C} at 4.3 3.05
  \put {x} at 1.5 -.5
  \put {y} at 3.9 -.4
  \put {z} at 1 -.2

\setsolid
\myarrow from  1.7 -.5 to 3.2 -.5 
\myarrow from 4 -.2 to 4.9 0.6
\myarrow from 1 0 to 1 3

\endpicture}

\caption{$P=[0,\frac{1}{2}]\times [0,\frac{1}{2}] \times [-\frac{1}{2},\frac{1}{2}]$ is a fundamental
  polyhedron for the action of $G$ on ${\mathbb R}^3$. In the quotient space $G \backslash {\mathbb R}^3$, 
  vertices with the same label
  are identified, as are edges with the same endpoints and the same shading. The four
  edges with both endpoints labelled $C$ are identified with the upward orientation. All faces
  with the same labels are identified in quotient space. Three edges (two in the quotient) have non-trivial
  isotropy, as indicated.}
\end{center}
\end{figure}


The last task remaining is to identify the gluings on the boundary of $P$.
First, we have that the top and bottom squares of $P$ are identified (via $T_z\in G$). Secondly, the two sides incident
to the $z$-axis get ``folded together'' by $\sigma\in G$ (which rotates the front face $\pi/2$ radians to the left side face). Finally,
the element $T_x \circ \sigma$ maps the hyperplane $y=1/2$ (containing the back face) to the hyperplane
$x=1/2$ (containing the right side face). This element takes the line $(0, 1/2, z)$ to the line $(1/2, 0, z)$, identifying together the 
corresponding edges of $P$. On the line of intersection of these two hyperplanes, the element acts by 
$(1/2, 1/2, z)\mapsto (1/2, 1/2, z + 1/4)$. These give us the identifications between the faces of $P$, allowing us to 
obtain the description of $G\backslash \mathbb R^3$ shown in Figure 1. 

\begin{example}
For the group $G := \hei (\mathbb Z) \rtimes \mathbb Z_4$ described above, we have that 
$\rank\Big(K_0\big( C^*_r(G)\big)\otimes \mathbb Q\Big) = 5$ and 
$\rank\Big(K_1\big( C^*_r(G)\big)\otimes \mathbb Q\Big) = 5$.
\end{example}

Before establishing this result, we note that this is consistent with the computation by L\"uck, who 
showed that $K_n\big( C^*_r(G)\big) \cong \mathbb Z^5$ for all $n$ (see \cite[Thm. 2.6]{Lu3}). This serves as a second
check on our algorithm, and is, to the best of our knowledge, the only example in the literature of an explicit computation 
for the topological $K$-theory of a $3$-orbifold group {\it with non-trivial torsion}.



\begin{figure}
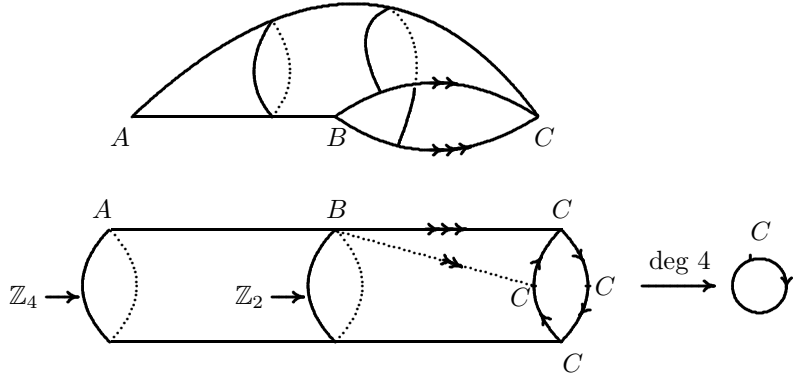

\begin{center}\hfil
\hbox{
\vbox{\beginpicture
\setcoordinatesystem units <1.2cm,1.5cm> point at -1 5
\setplotarea x from -.5 to 8, y from -.5 to 3.2
\linethickness=.7pt
%
\def\myarrow{\arrow <4pt> [.2, 1]} 
\setplotsymbol ({\circle*{.4}})
\plotsymbolspacing=.3pt        
\putrule from 0 0 to 5 0
\putrule from 0 1 to 5 1
\myarrow from 3.5 1 to 3.65 1 
\myarrow from 3.65 1 to 3.8 1
\myarrow from 3.8 1 to 3.95 1
\putrule from .25 2 to 2.5 2
\myarrow from 5.9 .5 to 6.7 .5
\put {deg 4} [b] at  6.3 .6
\circulararc 360 degrees from  7.5 .5 center at 7.2 .5 
\myarrow from 7.5 .5 to 7.5 .48
\plot 7.1 .77  7.1 .73 /
\put {$C$} [b] at  7.2 .9
\setdots <2pt>
\plot 2.53 .98   4.7 .5 /
\setsolid
\myarrow from 3.7 .729 to 3.8 .7
\myarrow from 3.8 .7 to 3.9 .671

\put {$A$} [b] at  -.1 1.1
\put {$B$} [b]  at  2.5  1.1
\put {$C$} [b] at  5 1.1
\put {$\mathbb{Z}_4$} [r] at  -.8 .4
\myarrow from -.7 .4 to -.35 .4
\put {$\mathbb{Z}_2$} [r]  at  1.7 .4
\myarrow from 1.8 .4 to 2.15 .4

\plot 4.68 .5  4.72 .5 /
\plot 5.28 .5  5.32 .5 /
\put {$C$} [tl] at  5 -.1
\put {$C$} [r] at  4.65 .4
\put {$C$} [l] at  5.36 .5

\put {$A$} [t] at  .1 1.9
\put {$B$} [t] at  2.5 1.9
\put {$C$} [t] at  4.8 1.9

\setquadratic
\plot  
.25 2   2.7 3    4.75 2 /
\plot  
2.5 2  3.6 2.3  4.75 2 /
\plot  
2.5 2   3.6 1.7  4.75 2 /
\myarrow from 3.55 2.3 to 3.7 2.3
\myarrow from 3.7 2.3 to 3.85 2.29
\myarrow from 3.55 1.7 to 3.7 1.7
\myarrow from 3.7 1.7 to 3.85 1.71
\myarrow from 3.85 1.71 to 4 1.73

\plot  
5 0   5.3 .5   5 1 /
\plot  
5 0   4.7 .5   5 1 /
\myarrow from 4.774 .74 to 4.784 .751
\myarrow from 5.24 .73 to 5.25 .72 
\myarrow from 4.784 .24 to 4.774 .25
\myarrow from 5.22 .25 to 5.21 .24

\setdots <2pt>
\plot  
2.5 0   2.8 .5   2.5 1 /
\setsolid
\plot  
2.5 0   2.2 .5   2.5 1  /

\setdots <2pt>
\plot  
0 0   .3 .5   0 1 /
\setsolid
\plot  
0 0   -.3 .5   0 1 /

\setdots <2pt>
\plot  
1.8 2  2 2.4  1.8 2.85 /
\setsolid
\plot  
1.8 2  1.6 2.4  1.8 2.85 /

\plot  
3  2.22  2.84  2.7   3.1 2.96 /
\setdots <2pt>
\plot  
3.2 1.75  3.4 2.5   3.1 2.96  /
\setsolid
\plot  
3.2 1.75  3.324 2.04   3.38 2.25 /

\endpicture}
}
\end{center}

\caption{Quotient space $G\backslash{\mathbb R}^3$. The four side faces of
  the polyhedron $P$ fold up into the two adjacent cylinders. On the right,
  the boundary circle of the cylinder gets attached to the circle by a degree
  4 map.  The top and bottom faces of $P$ get identified into a single
  square, which attaches to the cylinder as indicated. The two loops in the
  cylinder based at $A,B$ have isotropy ${\mathbb Z}_4$ and ${\mathbb Z}_2$
respectively. All remaining points have trivial isotropy.} 

\end{figure}


\begin{proof} We apply our algorithm, using the polyhedron $P$ described above. For the $\sim$ equivalence classes
on $F(G)$, we note that the quotient space $G\backslash \mathbb R^3$ has three vertices, one each with
stabilizer $\mathbb Z_4$ (vertex $A$), $\mathbb Z_2$ (vertex $B$), and the trivial group (vertex $C$). 
The edges joining distinct edges all have trivial stabilizer,
allowing us to identify all the identity elements together. We conclude that there are precisely five $\sim$ equivalence classes,
corresponding to the three non-trivial elements in the $\mathbb Z_4$ vertex stabilizer, the single non-trivial element in
the $\mathbb Z_2$ vertex stabilizer, and the equivalence class combining all the trivial elements. This gives us 
$\rank (H_0(\mathcal C)\otimes \mathbb Q) = 5$.

Next we consider the quotient space $G\backslash \mathbb R^3$. The faces of $P$ are pairwise identified, so the quotient
space is a closed manifold. Moreover, with respect to the induced orientation on $\partial P$, the identifications between
the faces are orientation {\it reversing}, so the quotient space is an orientable closed $3$-manifold. Lemma \ref{third-homology} 
gives us that $H_3(\mathcal C) \cong \mathbb Z$, and hence that $\rank (H_3(\mathcal C)\otimes \mathbb Q) = 1$. Note that,
as mentioned earlier, \cite[Lemma 2.4]{Lu3} shows that the quotient space is actually a $3$-sphere (but we do not need this
fact for our computation).

The quotient space has empty boundary, so $s=t=t^\prime =0$. The $2$-complex $Y$ is just the $2$-skeleton of the 
quotient space. This is the image of the boundary of $P$ after performing the required identifications. As such, $Y$
is constructed from two squares, a triangle, and a hexagon (see Figure 2). Note that the square corresponding to the front
face of $P$ (which also gets identified to the left face) folds up to a cylinder in $Y$, as its top and bottom edge get
identified together (left most cylinder in Figure 2). The union of the hexagon and triangle, forming the back face of $P$ 
(which also gets identified to the right face), similarly folds up to another cylinder in $Y$ (right most cylinder in Figure 2). 
The two cylinders attach together along a common boundary loop (image of the edge $BB$) to form a single long cylinder. 
At one of the endpoints, the cylinder attaches to a single
loop (image of the edge $CC$) by a degree four map. So ignoring for the time being the last square, we have a subcomplex
of $Y$ which deformation retracts to $S^1$ (as it coincides with the mapping cylinder of the degree four map of $S^1$). 
Up to homotopy, we conclude that $Y$ coincides with $S^1$, along with a single  square attached. The square comes 
from the top face of $P$ (which also gets identified with the bottom face), which, after composing with the homotopy to 
$S^1$, attaches to the $S^1$ via a degree one map of the boundary. This tells us that $Y$ is homotopy equivalent to a 
$2$-disk, and hence is contractible. By Corollary \ref{h2-term:thm}, we conclude that 
$\rank \big(H_2(\mathcal C)\otimes \mathbb Q\big) = 0$.

Finally, we compute the Euler characteristic of $\mathcal C$. We have three vertices, one each with stabilizer
$\mathbb Z_4$, $\mathbb Z_2$, and trivial. This gives an overall contribution of $+7$ to $\chi (\mathcal C)$. 
We have six edges, one with stabilizer $\mathbb Z_4$, one with stabilizer $\mathbb Z_2$, and the remainder
with trivial stabilizer. This contributes $-10$ to $\chi (\mathcal C)$. There are four faces with trivial stabilizer, contributing 
$+4$ to  to $\chi (\mathcal C)$. There is one $3$-cell with trivial stabilizer, contributing $-1$. Summing these
up, we see that $\chi(\mathcal C) = 7 -10 +4 -1 =0$. From Lemma \ref{h1-rank}, we see that 
$\rank \big(H_1(\mathcal C)\otimes \mathbb Q \big)=4$. Applying Lemma \ref{rational-expression}, 
we deduce that both the rational $K$-groups have rank $=5$, as claimed.
\end{proof}


\subsection{Hyperbolic reflection groups - I}\label{hyp-I-example}

Consider the
groups $\Lambda _n$, $n\geq 5$, given by the following presentation:

$$\Lambda _n:= \Bigg\langle
y, z, x_i, \hskip 5pt 1\leq i \leq n
\hskip 5pt \Bigg|  \hskip 5pt
\parbox{2.5in}
{\centerline{$y^2, z^2,$}

\centerline{$x_i^2,  (x_ix_{i+1})^2, (x_iz)^3, (x_iy)^3$, \hskip 5pt $1\leq i \leq n $}}
\Bigg\rangle$$
The groups $\Lambda _n$ are Coxeter groups, and the presentation given
above is in fact a Coxeter presentation of the group. 


\begin{figure}
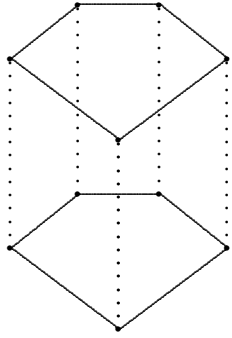
  
\label{graph}
\begin{center}

\vbox{\hbox{\beginpicture
  \setcoordinatesystem units <1.8cm,1.8cm> point at -.4 2.5
  \setplotarea x from -1 to 5, y from -.3 to 3

\def\smallbul{\hskip .8pt\circle*{2.2}}



  \setsolid
  \plot 2.5 0   1.7 .6  2.2 1   2.8 1    3.3 .6   2.5 0 /
  \plot 2.5 1.4  1.7 2  2.2 2.4  2.8 2.4   3.3 2  2.5 1.4 /

\setplotsymbol ({\circle*{1.2}}) 
\setdots

  \plot 2.5 0 2.5 1.4 /
  \plot 1.7 .6  1.7 2 /
  \plot 2.2 1  2.2 2.4 /
  \plot  2.8 1 2.8 2.4 /
  \plot  3.3 .6 3.3 2 /
  \plot   2.5 0 2.5 1.4 /
  \put {\smallbul} at 2.5 0
  \put {\smallbul} at 1.7 .6
  \put {\smallbul} at 2.2 1
  \put {\smallbul} at 2.8 1 
  \put {\smallbul} at 3.3 .6
  \put {\smallbul} at 2.5 1.4
  \put {\smallbul} at 1.7 2
  \put {\smallbul} at 2.2 2.4
  \put {\smallbul} at 2.8 2.4
  \put {\smallbul} at 3.3 2
  
  \endpicture}}
\caption{Hyperbolic polyhedron for $\Lambda _5$. Ordinary edges have internal dihedral angle $\pi/3$.
Dotted edges have internal dihedral angle $\pi/2$.}
\end{center}
\end{figure}


\begin{example}
For the groups $\Lambda _n$ whose presentations are given above,
\begin{enumerate}

\item the rank of $K_0\big( C^*_r(\Lambda_n)\big)\otimes \mathbb Q$ is equal to $3n+4$,

\item the rank of $K_1\big( C^*_r(\Lambda_n)\big)\otimes \mathbb Q$ is equal to $n+1$.

\end{enumerate}
\end{example}

\begin{proof}

The groups $\Lambda _n$ arise as hyperbolic reflection groups, with underlying polyhedron
$P$ the product of an $n$-gon with an interval. This polyhedron
has exactly two faces which are $n$-gons, and the dihedral angle along the edges of these two faces
is $\pi/3$. All the remaining edges have dihedral angle $\pi/2$. 
 An illustration of the polyhedron associated to the group $\Lambda _5$ is shown in Figure 3.
We will take the $\Lambda _n$ action
on $X:= \mathbb H^3$, with fundamental polyhedron $P$, and quotient space $X/\Lambda _n$
coinciding with $P$. Note that this action is a model for $\underline{E}\Lambda _n$, as finite subgroups $F$
of $\Lambda _n$ have non-empty fixed sets (the center of mass of any $F$-orbit will be a fixed point of
$F$), which must be convex subsets (and hence contractible). Both of these last statements are consequences
of the fact that the action is by isometries on a space of non-positive curvature.

Applying the argument detailed in Section \ref{algorithm}, we compute 
$\beta_0(\mathcal C)$ by counting equivalence classes on the set $F(\Lambda _n)$.
Since $X/\Lambda _n = P$, the set $F(\Lambda _n)$ consists of $2n$ copies of the group $S_4$. 
Each individual $S_4$ has five conjugacy classes, given by the possible cycle structures of elements,
with typical representatives: $e$, $(12)$, $(123)$, $(1234)$, $(12)(34)$. 
Next we consider how the edges identify the individual conjugacy classes to get the equivalence
classes for $\sim$. 

Firstly, all the individual
identity elements will be identified together, yielding a single $\sim$ class. So we will
henceforth focus on non-identity classes.
Each of the edges on the top $n$-gon has stabilizer $D_3 \cong S_3$, which has three conjugacy
classes, represented by $e$, $(12)$, $(123)$. Under the inclusion into each adjacent vertex 
stabilizers, representative elements for these classes map to representative elements with the
same cycle structure. So we see that all of the 3-cycles in the stabilizers of the vertices in the top
$n$-gon lie in the same $\sim$ class, and likewise for all of the 2-cycles. A similar analysis applies
to the vertices in the bottom $n$-gon. Finally, each vertical edge has stabilizer $D_2$, and under
the inclusion into the adjacent vertices, has image generated by the two permutations $(12)$ and
$(34)$ (and hence identifies {\it three} conjugacy classes together). Putting all this together, we 
see that the $\sim$ equivalence 
classes consist of:
\begin{itemize}
\item one class consisting of all the identity elements in the individual vertex groups, 
\item $n$ classes of elements of order $=2$, coming from the identification of cycles of the form $(12)(34)$
for each pair of vertices joined by a vertical edge,
\item one class of elements of order $=2$, coming from the cycles of the form $(12)$ in all vertex stabilizers,
\item two classes of elements of order $=3$, each coming from the cycles of the form $(123)$ in the top 
and bottom $n$-gon respectively, and
\item $2n$ classes of elements of order $=4$, each coming from the cycles of the form $(1234)$ in each individual
vertex stabilizer.
\end{itemize}
We conclude that the $\beta_0(\mathcal C)=\rank \big(H_0(\mathcal C)\otimes \mathbb Q\big) = 3n+4$.

Since our quotient space $X/\Lambda_n = P$ is not a closed orientable manifold, Lemma \ref{third-homology} tells us that
$H_3(\mathcal C) = 0$. To calculate $\beta_2(\mathcal C)=\rank \big(H_2(\mathcal C)\otimes \mathbb Q\big)$, 
we apply Corollary
\ref{h2-term:thm}. There is a single boundary 
component for $X/\Lambda_n=P$, which is orientable and even (it contains edges with stabilizer $D_2$), and contains
no edges with stabilizer $\mathbb Z_2$, so $s=1$, $t=0$, and $t^\prime=0$.
Also, there are no interior $2$-cells, and the single boundary component is of dihedral type, so $Y=\emptyset$.
By Corollary \ref{h2-term:thm}, we conclude that $\rank \big(H_2(\mathcal C)\otimes \mathbb Q\big) = 0$.

To calculate $\rank \big(H_1(\mathcal C)\otimes \mathbb Q\big)$, we need the Euler characteristic of the chain 
complex $\mathcal C$. There are $2n$ vertices, all with stabilizers $S_4$, which each have five conjugacy classes.
There are a total of $3n$ edges, $n$ of which have stabilizer $D_2$ (with four conjugacy classes), and $2n$ of which 
have stabilizer $D_3$ (with three conjugacy classes).
There are $n+2$ faces, with stabilizers $\mathbb Z_2$, which each have two conjugacy classes. 
There is one $3$-cell, with trivial stabilizer, with a single conjugacy class.
Putting this together, we have that
$$\chi (\mathcal C) = \big( 5(2n)\big) - \big(3(2n)+ 4(n)\big) + \big(2(n+2)\big) - 1 = 2n+3$$
Applying Lemma \ref{h1-rank}, we can now calculate:
$$\rank \big(H_1(\mathcal C)\otimes \mathbb Q\big) =  (3n+4) - (2n+3) = n+1$$
Finally, applying Lemma \ref{rational-expression}, we obtain the desired result.
\end{proof}


\subsection{Hyperbolic reflection groups - II}\label{hyp-II-example}

Next, let us consider a somewhat more complicated family of examples.  For
an integer $n\geq 2$, we consider the group $\G _n$, defined by the following
presentation:
$$\G _n:= \Bigg\langle x_1, \ldots , x_6   \hskip 5pt \Bigg|  \hskip 5pt
\parbox{3.3in}
{\centerline{$x_i^2, (x_1x_2)^n, (x_1x_5)^2, (x_1x_6)^2, (x_3x_4)^2, (x_2x_5)^2, (x_2x_6)^2$}

\centerline{$(x_1x_4)^3, (x_2x_3)^3, (x_4x_5)^3, (x_4x_6)^3, (x_3x_5)^3, (x_3x_6)^3$}}
\Bigg\rangle $$
Observe that the groups $\G_n$ are Coxeter groups, and that the presentation given
above is in fact a Coxeter presentation of the group. 

\vskip 10pt

\begin{example}
For the groups $\G _n$ whose presentations are given above, we have that:

$$\rank\Big(K_0\big( C^*_r(\G_n)\big)\otimes \mathbb Q\Big) = 
\begin{cases}
\frac{3}{2}(n-1)+12 & \text{$n$ odd,}\\
\frac{3}{2} n +14 & \text{$n$ even,}\\
\end{cases}
$$

$$
\rank\Big(K_1\big( C^*_r(\G_n)\big)\otimes \mathbb Q\Big) = 
\begin{cases}
3 & \text{$n$ odd,}\\
2 & \text{$n$ even.}\\
\end{cases}
$$

\end{example}

\begin{proof}
To verify the results stated in this example, we first observe that the Coxeter
groups $\G _n$ arise as hyperbolic reflection groups, with underlying polyhedron
$P$ a combinatorial cube. The geodesic 
polyhedron associated to $\G _n$ is shown in Figure 4.  Again, we set $X:= \mathbb H^3$, 
with fundamental polyhedron $P$, and quotient space $X/\G _n$ coinciding with $P$. As in the
previous example, $X$ is a model for $\underline{E}G$.


\begin{figure}
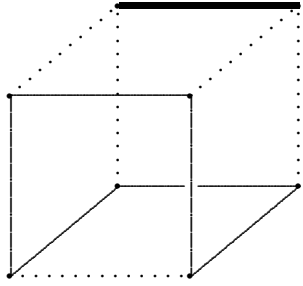
   
\label{graph}
\begin{center}

\vbox{\beginpicture
  \setcoordinatesystem units <1.6cm,1.6cm> point at -.4 2.5
  \setplotarea x from -.8 to 3, y from -.6 to 2.5

  \def\smallbul{\hskip .8pt\circle*{2.2}}
  \linethickness=.7pt



  \setsolid
  \plot 2.2 0  2.2 1.5  3.7 1.5  3.7 0  4.6 .75  3.76 .75 /
  \plot 3.64 .75   3.1 .75  2.2 0 /

\linethickness=2pt
  \putrule from 3.1 2.25 to 4.6 2.25 

\setdots
\setplotsymbol ({\circle*{1.2}}) 

  \plot 2.2 0  3.7 0  /
  \plot 3.1 .75   3.1 2.25  /
  \plot 4.6 .75  4.6 2.25  /
  \plot 2.2 1.5  3.1 2.25  /
  \plot 3.7 1.5   4.6 2.25  /
  \put {\smallbul} at 2.2 0 
  \put {\smallbul} at 3.7 0
  \put {\smallbul} at 2.2 1.5
  \put {\smallbul} at 3.7 1.5
  \put {\smallbul} at 3.1 .75
  \put {\smallbul} at 4.6 .75
  \put {\smallbul} at 3.1 2.25
  \put {\smallbul} at 4.6 2.25
  
  \endpicture}

\caption{Hyperbolic polyhedron for $\Gamma _n$. Ordinary edges have internal dihedral angle $\pi/3$.
Dotted edges have internal dihedral angle $\pi/2$. The thick edge has internal dihedral angle $\pi/n$.}
\end{center}
\end{figure}


To apply our procedure, we start by considering the equivalence relation $\sim$ 
on the set $F(\G _n)$. Out of the eight vertices of the cube $P$, six have stabilizer
isomorphic to $S_4$, while the remaining two have stabilizer $D_n\times \mathbb Z_2$.
We will think of $D_n$ as the symmetries of a regular $n$-gon, and let $r_0$, $r_1$ denote the reflection 
in a vertex, and in the midpoint of an adjacent side respectively (so $r_0, r_1$ are the standard
Coxeter generators for $D_n$).
Recall that the number of conjugacy classes of $D_n$ depends
on the parity of $n$: each rotation $\phi$ is only conjugate to its inverse $\phi^{-1}$, while the reflections
$r_i$ fall into one or two conjugacy classes, depending on whether $n$ is odd or even.
Crossing with $\mathbb Z_2$, each of these conjugacy class in $D_n$
gives rise to two conjugacy classes in $D_n \times \mathbb Z_2$: the image class
under the obvious inclusion $D_n\hookrightarrow D_n\times \mathbb Z_2$, and its
``flipped" image, obtained by composing with the non-trivial element $\tau$ in the 
$\mathbb Z_2$-factor. Next, we need to see how conjugacy classes in the individual vertex 
stabilizers get identified together by the edge stabilizers. After performing these
identifications, we obtain that the $\sim$ equivalence classes consist of:
\begin{itemize}
\item one class consisting of all the identity elements in the individual vertex groups, 
\item six classes of elements of order $=4$, each coming from the cycles of the form $(1234)$ in the six 
individual $S_4$ vertex stabilizers,
\item one class of elements of order $=3$, coming from the cycles of the form $(123)$ in the six $S_4$
vertex stabilizers (these classes get identified together via the edges with stabilizer $D_3$),
\item one class of elements of order $=2$, comprised from the cycles of the form $(12)$ in the six $S_4$
vertex stabilizers (identified via the edges with stabilizer $D_3$), along with the the three elements
of the form $(r_0, 1), (r_1, 1), (1, \tau)$ in the two vertices with stabilizer $D_n\times \mathbb Z_2$ 
(identified via the edges with stabilizer $D_2$),
\item one class of elements of order $=2$, consisting of the elements of cycle form $(12)(34)$ in the
two $S_4$ vertex stabilizers which are joined together by an edge with stabilizer
$D_2$ (which identifies these elements together),
\item two or four classes (according to parity of $n$), coming from the two elements of the form
$(r_0, \tau)$ or $(r_1, \tau)$ in the two vertices with stabilizer $D_n\times \mathbb Z_2$ (these
two elements lie in the same conjugacy class when $n$ odd), which are each identified to elements
with cycle form $(12)(34)$ in one of the two adjacent $S_4$ vertex stabilizers,
\item $n-1$ or $n$ conjugacy classes (according to $n$ odd or even respectively), coming from elements 
of the form $(\phi_i , \tau)$ in each of the two vertices with stabilizer $D_n\times \mathbb Z_2$, and
\item $(n-1)/2$ or $n/2$ conjugacy classes (according to $n$ odd or even respectively), coming from the
elements of the form $(\phi_i, 1)$ in the two vertices with stabilizer $D_n\times \mathbb Z_2$ (the elements
in the two copies get identified together via the edge with stabilizer $D_n$).
\end{itemize}
Summing this up, we find that $\rank \big(H_0(\mathcal C)\otimes \mathbb Q\big)$ is $\frac{3}{2}(n-1) + 12$
if $n$ is odd, and $\frac{3}{2}n + 14$ if $n$ is even.

The quotient space $X/\Gamma _n = P$ is a $3$-manifold with non-empty boundary, so Lemma \ref{third-homology} 
gives us that $H_3(\mathcal C) = 0$. The only boundary component is orientable and even, and contains no edges
with stabilizer $\mathbb Z_2$, so $s=1$ and $t=t^\prime=0$. Moreover, there are no interior faces, so $Y=\emptyset$.
By Corollary \ref{h2-term:thm}, we conclude that $\rank \big(H_2(\mathcal C)\otimes \mathbb Q\big) = 0$.

Next, let us calculate the rank of $H_1(\mathcal C)\otimes \mathbb Q$. To do this, we first compute
the Euler characteristic $\chi(\mathcal C)$. We have six vertices, four with stabilizer $S_4$ (having five
conjugacy classes), and two with stabilizer $D_n\times \mathbb Z_2$ (having either $n+3$ or $n+6$
conjugacy classes, depending on whether $n$ is odd or even). There are twelve edges, six with stabilizer
$D_3$ (with three conjugacy classes), five with stabilizer $D_2$ (with four conjugacy classes), and one
with stabilizer $D_n$ (with $(n+3)/2$ or $(n+6)/2$ conjugacy classes, depending on whether $n$ is odd or even).
There are six faces, each with stabilizer $\mathbb Z_2$ (with two conjugacy classes each). Finally, there is 
one $3$-cell with trivial stabilizer. Taking the alternating sum, we obtain that the Euler characteristic is:
$$
\chi (\mathcal C)= 
\begin{cases}
\frac{3}{2}(n-1)+9 & \text{$n$ odd,}\\
\frac{3}{2}n + 12 & \text{$n$ even.}\\
\end{cases}
$$
From Lemma \ref{h1-rank}, the difference between $\chi(\mathcal C)$ and the rank of 
$H_0(\mathcal C)\otimes \mathbb Q$ yields the rank of
$H_1(\mathcal C)\otimes \mathbb Q$, giving us that the latter is either $3$ or $2$ according
to whether $n$ is odd or even. Applying Lemma \ref{rational-expression}, we obtain 
the desired result.
\end{proof}


\subsection{Crystallographic group}\label{Crystal-example}
Our next example is taken from the work of Farley and Ortiz \cite{FO}. 
Consider the lattice $L\subset \mathbb R^3$ generated by the three vectors
$$ \mathbf{v}_{1} = \left  (\begin{smallmatrix} 2/3 \\  -1/3 \\ 2/3 \end{smallmatrix} \right), \quad
\mathbf{v}_{2} = \left(\begin{smallmatrix} 1 \\ -1 \\ 0 \end{smallmatrix} \right), \quad  
\mathbf{v}_{3} = \left( \begin{smallmatrix} 0 \\ -1 \\ 1 \end{smallmatrix} \right),$$
and let $G= \Sym (L)$ denote the subgroup of $\textup{Isom}(\mathbb R^3)$ which maps $L$
to itself. The group $G$ is one of the seven maximal split $3$-dimensional crystallographic groups,
and is discussed at length in \cite[Section 6.7]{FO}.

A polyhedral fundamental domain $P$ for the $G$-action on $\mathbb R^3$ is provided in Figure 5.
Next we describe the stabilizers of the various faces, edges, and vertices of $P$ (given in terms of the labeling
in Figure 5).

\vskip 5pt


\begin{figure}[h]
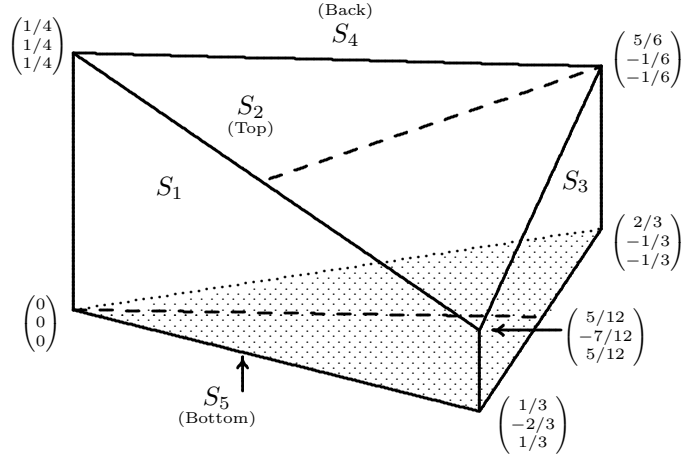
 
    
\begin{center}\hfil
\hbox{
\vbox{
  \beginpicture
\setcoordinatesystem units <.9cm,.9cm> point at -.2  4.2
\setplotarea x from -.5 to 8.2, y from -2 to 4.2
\linethickness=.7pt
%
 \def\myarrow{\arrow <4pt> [.2, 1]} 
%
\setplotsymbol ({\circle*{.4}})
\plotsymbolspacing=.3pt        
\plot 0 0  6 -1.5   7.8 1.2   7.8 3.6   0 3.8  0 0 /
\plot 0 3.8  6 -.3  7.8 3.6 /
\plot 6 -1.5  6 -.3 /
\setplotsymbol ({\circle*{.2}})
\setdashes
\plot 2.9 1.94   7.8 3.6 /
\plot 0 0  6.9 -.1 /
\setdots <3pt>   
\plot 0 0 7.8 1.2 /
\put {\tiny $\left(\begin{matrix}0 \\ 0 \\ 0\end{matrix}\right)$} [r] at -.2 -.2
\put {\tiny $\left(\begin{matrix}5/12\\-7/12\\5/12\end{matrix}\right)$} 
[l] at 7.2 -.4
\put {\tiny $\left(\begin{matrix}1/3 \\ -2/3 \\ 1/3\end{matrix}\right)$} [l]
at 6.2 -1.7
\put {\tiny $\left(\begin{matrix}2/3 \\ -1/3 \\ -1/3\end{matrix}\right)$}
[l] at 7.9 1
\put {\tiny $\left(\begin{matrix}5/6 \\ -1/6 \\ -1/6\end{matrix}\right)$}
[l] at 7.9 3.7
\put {\tiny $\left(\begin{matrix}1/4 \\ 1/4 \\ 1/4\end{matrix}\right)$} [r]
at -.1 3.9
\put {$S_1$} at 1.4 1.8
\put {$S_2$} at 2.6 3
\put {\tiny (Top)} [t] at 2.6 2.8
\put {$S_3$} at 7.4 1.9 
\put {$S_4$} [b] at 4 3.9
\put {\tiny (Back)} [b] at 4 4.3
\put {$S_5$} at 2.1 -1.3
\put {\tiny (Bottom)} [t] at 2.1 -1.5
\setshadegrid span <2pt>
\hshade -1.5 6 6   0 0 7 /  
\hshade -.1 0 7   1.2 7.8 7.8 /  

\setsolid
\myarrow from  7.2 -.3 to 6.2 -.3 
\myarrow from 2.5 -1.2 to 2.5 -.7
\endpicture}
}
\hfil

\end{center}


\caption{The polyhedron pictured here is an exact convex compact fundamental
  polyhedron for the action of $G$ on ${\mathbb R}^3$. The dashed lines
  represent axes of rotation (through 180 degrees) for certain elements of
  $G$.  Note that the base of the figure is an equilateral triangle,
  but the top is only isosceles.}
\end{figure}


\noindent \underline{Face stabilizers:} The two triangles at the top (collectively labelled by $S_2$), and the
two triangles at the bottom (labelled by $S_5$) have trivial stabilizer. The three quadrilateral sides ($S_1$, $S_3$,
and $S_4$) each have stabilizer $\mathbb Z_2$, generated by the reflection in the $2$-plane extending the corresponding
side. 

\vskip 5pt

\noindent \underline{Edge stabilizers:} The three vertical edges in Figure 5 each have stabilizer $D_3$, generated by
the reflections in the two incident faces. The two dotted edges (in the middle of the faces $S_2$ and $S_5$) have
stabilizer $\mathbb Z_2$, generated by a rotation by $\pi$ centered on the edge. All remaining edges have stabilizer
$\mathbb Z_2$, generated by the reflection in the (unique) incident face whose isotropy is non-trivial. Note that, when
one passes to the quotient space $X/G$, the two triangles in the top face $S_2$ get identified together by the $\pi$-rotation
in the dotted line (and similarly for the two triangles in the bottom face $S_5$).

\vskip 5pt

\noindent \underline{Vertex stabilizers:} The two vertices $(0,0,0)$ and $(5/6, -1/6, -1/6)$ have stabilizer 
$D_3\times \mathbb Z_2$. The two vertices $(1/4, 1/4, 1/4)$ and $(2/3, -1/3, -1/3)$ have stabilizer $D_3$.
Finally, the two vertices $(1/2, 1/2, 0)$ and $(1/3, -1/6, 1/3)$, the midpoints of the edges at which the dotted
lines terminate, have stabilizer $D_2$. The remaining vertices
of $P$ are in the same orbit as one of the six described above.

\vskip 10pt

\begin{example}
For the split crystallographic group $G$ described above, we have that 
$\rank\Big(K_0\big( C^*_r(\G_n)\big)\otimes \mathbb Q\Big) = 12$ and 
$\rank\Big(K_1\big( C^*_r(\G_n)\big)\otimes \mathbb Q\Big) = 0$.
\end{example}

\begin{proof}
We apply our algorithm, using the polyhedron $P$ above. Our first step is to consider the $\sim$
equivalence relation on the set $F(G)$. The vertex and edge stabilizers for $P$ have been described
above, and the $\sim$ equivalence classes are given as follows:
\begin{itemize}
\item one class consisting of all the identity elements in the individual vertex groups, 
\item one class consisting of all the elements of order $3$ in the individual vertex groups (these occur
in the four vertices with stabilizer $D_3$ or $D_3\times \mathbb Z_2$, and are identified together via three
consecutive edges with stabilizer $D_3$),
\item one class of elements of order $2$, consisting of elements of order two in the vertex groups isomorphic
to $D_3$, along with elements of order two in the canonical $D_3$-subgroup within the vertex groups
isomorphic to $D_3\times \mathbb Z_2$ (these are identified together via the three consecutive
edges with stabilizer $D_3$), and the elements of the form $(1,0)$ in the two vertex groups isomorphic
to $D_2\cong \mathbb Z_2\times \mathbb Z_2$ (identified together via the edges $S_1\cap S_2$ and
$S_3\cap S_5$),
\item two classes of elements of order $2$, coming from each of the two dotted edges: the rotation by $\pi$ in the edge
identifies the element $(0,1)$ in one endpoint (vertex with stabilizer $D_2\cong \mathbb Z_2\times \mathbb Z_2$) with
the element which is a product of a reflection in $D_3$ with a reflection in $\mathbb Z_2$ in the other endpoint 
(vertex with stabilizer $D_3\times \mathbb Z_2$),
\item six remaining classes, two each in the vertices with stabilizer $D_3\times \mathbb Z_2$ and one each in those with
stabilizer $D_2$ (these classes aren't identified to any others via the edges).
\end{itemize}
Summing this up, we see that $\rank (H_0(\mathcal C)\otimes \mathbb Q)=11$.

Next, we note that the quotient space $X/G$ is obtained from the polyhedron $P$ by ``folding up'' the top and bottom
triangle along the dotted lines, resulting in $\mathbb D^3$, a $3$-manifold with non-empty boundary. Lemma \ref{third-homology} 
gives us that $H_3(\mathcal C) = 0$. The only boundary component is orientable and odd, and contains edges
with stabilizer $\mathbb Z_2$, so $s=t=0$ and $t^\prime=1$. The $2$-complex $Y$ clearly deformation retracts
to the boundary $S^2$, so $\beta_2(Y)=1$. 
By Corollary \ref{h2-term:thm}, we conclude that $\rank \big(H_2(\mathcal C)\otimes \mathbb Q\big) = 1$.

Next, we calculate the rank of $H_1(\mathcal C)\otimes \mathbb Q$. As usual, we first calculate
the Euler characteristic $\chi(\mathcal C)$. We have six vertices, two with stabilizer $D_2$ (having four
conjugacy classes), two with stabilizer $D_3$ (having three
conjugacy classes), and two with stabilizer $D_3\times \mathbb Z_2$  (having six
conjugacy classes), giving an overall contribution of $+26$. There are nine edges, six with stabilizer
$\mathbb Z_2$ (with two conjugacy classes), and three with stabilizer $D_3$ (with three conjugacy classes),
giving a contribution of $-21$.
There are five faces, three with stabilizer $\mathbb Z_2$ (with two conjugacy classes each), and two with
trivial stabilizer (with one conjugacy class each), giving a contribution of $+8$. There is 
one $3$-cell with trivial stabilizer, contributing a $-1$. Summing up these contributions, we obtain that 
the Euler characteristic is
$ \chi (\mathcal C)= 26 - 21 + 8 -1=12$.
From Lemma \ref{h1-rank}, we see that the rank of 
$H_1(\mathcal C)\otimes \mathbb Q$ is $=0$. Applying Lemma \ref{rational-expression}, we obtain 
the desired result.
\end{proof}


\section{Concluding remarks}

The examples in the previous section were chosen to illustrate our algorithm on several different types of 
smooth $3$-orbifold groups. As the reader can see, our algorithm is quite easy to apply, once one has a good
description of the orbit space $G \backslash X$. There are several natural directions for further work. 

For instance, in Section \ref{Crystal-example}, we applied our algorithm to a specific $3$-dimensional crystallographic
group. It is known that, in dimension $=3$, there are precisely $219$ crystallographic groups up to isomorphism. One
could in principle apply our algorithm to produce a complete table of the rational $K$-theory groups of all $219$
groups. The essential difficulty in doing this lies in finding some convenient, systematic way to identify polyhedral 
fundamental domains for each of these groups. For the $73$ {\it split} 
crystallographic groups, such fundamental domains can be found in the forthcoming paper of Farley and Ortiz \cite{FO}.

Another reasonable direction would be to focus on uniform arithmetic lattices $\G$ in the Lie group 
$PSL_2(\mathbb C) \cong \textup{Isom}^+(\mathbb H^3)$. One could try to analyze the relationship (if any) between
the rational $K$-theory of such a $\G$ and the underlying arithmetic structure. Again, the difficulty here lies
in finding a good description of the polyhedral fundamental domain for the action (in terms of the arithmetic
data).

In a different direction, one can consider {\it hyperbolic reflection groups}. These are groups generated by reflections
in the boundary faces of a geodesic polyhedron $P \subset \mathbb H^3$. In this context, the polyhedron $P$ serves
as a polyhedral fundamental domain for the action, so one can readily apply our algorithm to compute the rational
$K$-theory of the corresponding group (see the examples in Sections \ref{hyp-I-example} and \ref{hyp-II-example}). 
One could try, in this special case, to refine our algorithm to produce expressions
for the {\it integral} $K$-theory groups, in terms of the combinatorial data of the polyhedron $P$. This is the subject of an 
ongoing collaboration of the authors.

\section*{Acknowledgments}

\vskip 10pt

The authors would like to thank Dick Canary, Ian Leary, Wolfgang L\"uck, Guido Mislin, Peter Scott, and Alain Valette for helpful comments.
The first author was partially supported by the NSF, under grant DMS-0906483, and by an Alfred P. Sloan Research Fellowship. The second author was partially supported by the NSF, under grant DMS-0805605.
The third author was supported by the EPSRC grant EP/G059101/1.

\section*{Appendix A}\label{appendix-A}

\vskip 10pt

In this Appendix, we provide the details for the computations used in some of the proofs in Section \ref{h2:section}.
Let $n \ge 2$ be an integer and $D_n$ be the dihedral group with presentation
\[
	D_n = \langle s_1, s_2 \; | \; s_1^2=s_2^2=(s_1s_2)^n \rangle.
\]
We will compute the map 
\begin{equation}\label{eqn3}
	\varphi \colon R_\mathbb{C}(\mathbb{Z}_2) \oplus R_\mathbb{C}(\mathbb{Z}_2) \longrightarrow R_\mathbb{C}(D_n)
\end{equation}
given by induction between representation rings with respect to the subgroups $\langle s_1 \rangle$ and $\langle s_2 \rangle$ 
of $D_n$, both isomorphic to $\mathbb{Z}_2$, and opposite orientations. That is, $\varphi(\rho, \tau) = (\rho \uparrow) - (\tau \uparrow)$, 
where `$\uparrow$' means induction between the corresponding groups.\\

Recall from the main text (see Section \ref{h2:section}, particularly Lemma \ref{bdry-dihedral-edge}) 
that if $e$ is a boundary edge with stabilizer $D_n$ and $\sigma_1$ and $\sigma_2$ are incident boundary 
faces, then $K_0(C^\ast_r (G_{\sigma_i})) \cong R_\mathbb{C}(\mathbb{Z}_2)$ and the relevant part of the Bredon chain complex 
at the edge $e$ is the map given in equation (\ref{eqn3}).\\

The character table for $D_n$ is given by
\[
	\begin{array}{c|cc}
		 D_n & (s_1s_2)^r & s_2(s_1s_2)^r\\ 
		 \hline 
		 \chi_1 & 1 &  1 \\
		 \chi_2 & 1 & -1 \\
		 \widehat{\chi_3} & (-1)^r &  (-1)^r \\
		 \widehat{\chi_4} & (-1)^r &  (-1)^{r+1} \\
		 \phi_p & 2 \cos\left( \frac{2 \pi p r}{m} \right) &  0 \\
	\end{array}
\]
\noindent where $0 \le r \le n-1$, $p$ varies between 1 and $n/2-1$ if $n$ is even or $(n-1)/2$ if $n$ is odd and the 
hat $\ \widehat{ }\ $ denotes a character which appears only when $n$ is even. \\
The character table for $\mathbb{Z}_2$ is given by
\[
	\begin{array}{c|rr}
		 \mathbb{Z}_2 & e & s_i\\ 
		 \hline 
		 \rho_1 & 1 & 1\\
		 \rho_2 & 1 & -1
	\end{array}
\] 

To compute the induction homomorphism we will use Frobenius reciprocity. We first do the case $\langle s_1 \rangle$. 
The characters of $D_n$ restricted to this subgroup are
\[
	\begin{array}{c|rr}
		  & e & s_1 \\ 
		 \hline 
		 \chi_1\downarrow & 1 & 1\\
		 \chi_2\downarrow & 1 & -1\\
		 \widehat{\chi_3}\downarrow & 1 & -1\\
		 \widehat{\chi_4}\downarrow & 1 & 1\\
		 \phi_p\downarrow & 2 & 0
	\end{array}
\]
Multiplying with the rows of the character table of $ \langle s_1 \rangle \cong C_2$ we obtain the induced representations
\[
	\begin{array}{rcl}
			\rho_1 \uparrow & = & \chi_1 + \widehat{\chi_4} + \sum \phi_p ,\\
			\rho_2 \uparrow & = & \chi_2 + \widehat{\chi_3} + \sum \phi_p .
	\end{array}
\]
The case $\langle s_2 \rangle$ is analogous, but note that the characters 3 and 4 must be interchanged in the even case:
\[
	\begin{array}{c|rr}
		  & e & s_j \\ 
		 \hline 
		 \chi_1 & 1 & 1\\
		 \chi_2 & 1 & -1\\
		 \widehat{\chi_3} & 1 & 1\\
		 \widehat{\chi_4} & 1 & -1\\
		 \phi_p & 2 & 0
	\end{array}
\]
and
\[
	\begin{array}{rcl}
			\rho_1 \uparrow & = & \chi_1 + \widehat{\chi_3} + \sum \phi_p ,\\
			\rho_2 \uparrow & = & \chi_2 + \widehat{\chi_4} + \sum \phi_p .
	\end{array}
\]
As maps of free abelian groups we obtain
\begin{eqnarray*}
	\mathbb{Z}^2 &\to & \mathbb{Z}^{c(D_n)} \\
	(a,b) &\mapsto & (a,b,\widehat{b},\widehat{a},a+b,\ldots , a+b) \quad \text{for } \langle s_1 \rangle \hookrightarrow D_n,\\
	(c,d) &\mapsto & (c,d,\widehat{c},\widehat{d},c+d,\ldots , c+d) \quad \text{for } \langle s_2 \rangle \hookrightarrow D_n.
\end{eqnarray*}
Finally, the map $\varphi$ above is
\begin{eqnarray*}
	R_{\mathbb{C}} \left( \mathbb{Z}_2 \right) \oplus  
		R_{\mathbb{C}} \left( \mathbb{Z}_2 \right) \cong \mathbb{Z}^2 \oplus \mathbb{Z}^2
		& \to & \mathbb{Z}^{c( D_n )} \cong 
		R_{\mathbb{C}} \left( D_n \right)\\
	(a,b,c,d) & \mapsto & (a-c, b-d, \widehat{b-c}, \widehat{a-d}, S,\ldots, S)\nonumber
\end{eqnarray*}
where $S = a+b-c-d$. 

\vskip 10pt

As an immediate consequence of this computation, we see that if the element $\langle k, k, \ldots , k\rangle$
lies in the image of $\phi$, then one must have that:
$$a - c = k = S = a + b - c - d.$$
Subtracting $a-c$ from both sides, we deduce that $0 = b - d = k$. In other words, the image of $\phi$ intersects
the subgroup $\mathbb Z \cdot \langle 1, 1, \ldots , 1\rangle$ only in the zero vector (as was stated in Lemma
\ref{bdry-dihedral-edge}).

\vskip 10pt

Another consequence is that it is easy to identify elements in the kernel of $\phi$. The equation 
$$0=(a-c, b-d, \widehat{b-c}, \widehat{a-d}, S,\ldots, S)$$
forces $a=c$ and $b=d$. If in addition, $n$ is even, then we also have $a=d$, and hence all terms must
be equal. This was used in the arguments for both Lemma \ref{Z-term:lemma} and Lemma \ref{phi-term:lemma}.

\section*{Appendix B}\label{appendix-B}

\vskip 10pt

In this Appendix we compute the $2^\text{nd}$ Betti number of the 3-manifolds $M_f$ appearing in the Remark at the end of
Section \ref{torsion-free-examples}. 
The manifold $M_f$, as a mapping torus, fibers over $S^1$ with fiber $T^2$. For this fibration, the Leray-Serre spectral 
sequence gives
\[
	E^2_{pq} = H_p(S^1,H_q(T^2)) \Rightarrow H_{p+q}(M_f).
\]
Since $S^1$ is 1-dimensional, $E^2_{p,q}=0$ unless $p=0,1$. The differentials have bidegree $(-2,1)$ so the spectral 
sequence already collapses at the $E^2$-page. This implies that
\[
	H_2(M_f) \cong E^2_{0,2} \oplus E^2_{1,1} \cong H_0(S^1,H_2(T^2)) \oplus H_1(S^1,H_1(T^2)).
\]
Recall that this is not ordinary homology but rather homology with local coefficient system given by the homology of the fiber.

The homology group $H_0(S^1,H_2(T^2))$ is obtained from the chain complex
\[
	\xymatrix{0 \ar[r] & \mathbb Z \ar[r]^-{\text{Id}-f_*} & \mathbb Z \ar[r] & 0}
\]
where $f_*\colon \mathbb Z \to \mathbb Z$ is the map induced by the action of the gluing map $f$ on the local 
coefficient $\mathbb Z = H_2(T^2)$.
If $\det(\alpha)=1$, $f$ is orientation preserving and hence $f_* = \text{Id}$. This implies $H_0(S^1,H_2(T^2)) \cong \mathbb Z$.
If $\det(\alpha)=-1$, $f$ is orientation reversing and hence $f_* = -\text{Id}$. This implies $H_0(S^1,H_2(T^2)) \cong \mathbb Z_2$.

The homology group $H_1(S^1,H_1(T^2))$ is obtained from the chain complex
\[
	\xymatrix{0 \ar[r] & \mathbb Z^2 \ar[r]^-{\text{Id}-f_*} & \mathbb Z^2\ar[r] & 0}
\]
where now $f_*:\mathbb Z^2 \to \mathbb Z^2$ is induced by the action of the gluing map $f$ on the local coefficient 
$\mathbb Z^2 = H_1(T^2)$. Note that by construction $f$ acts on $\pi_1(T^2) \cong H_1(T^2)$ via the automorphism 
$\alpha$. So the map above is $\text{Id}-\alpha$ and hence $H_1(S^1,H_1(T^2)) \cong \ker(\text{Id}-\alpha)$. 
Suppose that $\alpha = \left( \begin{smallmatrix} a & b \\ c & d\end{smallmatrix} \right)$. 
Then $\text{Id} - \alpha = \left( \begin{smallmatrix}1- a & -b \\ -c & 1-d\end{smallmatrix} \right)$. 
The kernel of this map has dimension 2 if and only if $\text{Id}-\alpha=0$, that is, $\alpha=\text{Id}$. 
The dimension is at least 1 if and only if the determinant is zero, that is,
\[
	(1-a)(1-d) = bc \;\Leftrightarrow\; 1 - \text{tr}(\alpha)+ad=bc \;\Leftrightarrow \; 1 + \det(\alpha) = \text{tr}(\alpha).
\]
 This occurs if and only if $\det(\alpha)=1$ and $\text{tr}(\alpha)=2$, or $\det(\alpha)=-1$ and $\text{tr}(\alpha)=0$.
Altogether, this gives us
\[
	\beta_2(M_f) = 
\begin{cases}
	3 & \text{if } \alpha = \text{Id},\\
	2 & \text{if } \det(\alpha)=1, \text{tr}(\alpha)=2,  \alpha \neq \text{Id},\\
	1 & \text{if } \det(\alpha)=1, \text{tr}(\alpha)\neq 2,\\
	1 & \text{if } \det(\alpha)=-1, \text{tr}(\alpha)= 0,\\
	0 & \text{if } \det(\alpha)=-1, \text{tr}(\alpha)\neq 0.\\
\end{cases}
\]


\end{document}